\documentclass{amsart}
\usepackage[utf8]{inputenc}
\usepackage{amsmath, amsthm, amssymb, amsfonts, mathtools, bbold, stmaryrd, latexsym, enumerate}    
\usepackage{wasysym}
\usepackage{dsfont}

\pdfoutput=1

\usepackage[dvipsnames]{xcolor}
\usepackage{stackengine}
\usepackage{float}
\usepackage{enumitem}
\usepackage{tikz}
\usepackage{comment}

\usepackage[colorlinks]{hyperref}
\hypersetup{
    colorlinks={true},
    linkcolor={BrickRed},
    citecolor={ForestGreen},
    urlcolor={Blue}
}
\usepackage[nameinlink]{cleveref}

\usepackage{nicematrix}
\NiceMatrixOptions
  {
    code-for-first-col = \scriptstyle ,
    code-for-first-row = \scriptstyle 
  }

\usepackage[parfill]{parskip}
\usepackage[margin=1in]{geometry}

 \newtheorem{theorem}{Theorem}[section]
 \newtheorem{corollary}[theorem]{Corollary}
 \newtheorem{lemma}[theorem]{Lemma}
 \newtheorem{proposition}[theorem]{Proposition}
  \newtheorem{claim}{Claim}
 \newtheorem{conjecture}{Conjecture}
 \newtheorem{question}{Question}

 \theoremstyle{definition}
 \newtheorem{definition}[theorem]{Definition}
 \theoremstyle{remark}
 \newtheorem{remark}[theorem]{Remark}
 \newtheorem{example}[theorem]{Example}
 \newtheorem*{notation}{\it Notation}

\numberwithin{equation}{section}

\newcommand{\Hrk}{\mathrm{Hrk}}
\newcommand{\rk}{\operatorname{rk}}

\newcommand{\Sym}{\mathrm{Sym}}

\newcommand{\Trop}{\operatorname{Trop}}
\newcommand{\Newt}{\operatorname{Newt}}

\newcommand{\diag}{\operatorname{diag}}

\newcommand{\calN}{\mathcal{N}}

\newcommand{\calX}{\mathcal{X}}

\newcommand{\bbC}{\mathbb{C}}

\newcommand{\bbP}{\mathbb{P}}
\newcommand{\bbR}{\mathbb{R}}
\newcommand{\bbT}{\mathbb{T}}

\newcommand{\bbZ}{\mathbb{Z}}

\newcommand{\bfn}{\mathbf{n}}
\newcommand{\bfd}{\mathbf{d}}
\newcommand{\bfr}{\mathbf{r}}

\newcommand{\bfx}{\mathbf{x}}
\newcommand{\bfy}{\mathbf{y}}
\newcommand{\bfY}{\mathbf{Y}}
\newcommand{\bfv}{\mathbf{v}}
\newcommand{\bfi}{\mathbf{i}}
\newcommand{\bfj}{\mathbf{j}}
\newcommand{\bfw}{\mathbf{w}}
\newcommand{\bfu}{\mathbf{u}}
\newcommand{\bfe}{\mathbf{e}}
\newcommand{\bfz}{\mathbf{z}}

\title{Hadamard ranks of algebraic varieties}

\author[D. Antolini]{Dario Antolini}
\author[G. Mont\'ufar]{Guido Mont\'ufar}
\author[A. Oneto]{Alessandro Oneto}

\address{Dario Antolini, Alessandro Oneto - Department of Mathematics, University of Trento, Via Sommarive, 14 - 38123 Povo (Trento), Italy}
\email{dario.antolini-1@unitn.it, alessandro.oneto@unitn.it}

\address{Guido Mont\'ufar - Departments of Mathematics and Statistics \& Data Science, University of California, Los Angeles Los Angeles, CA 90095}
\email{montufar@math.ucla.edu}

	
\begin{document}

\begin{abstract}
    Motivated by the study of decompositions of tensors as Hadamard products (i.e., coefficient-wise products) of low-rank tensors, we introduce the notion of {\it Hadamard rank} of a given point with respect to a projective variety: if it exists, it is the smallest number of points in the variety such that the given point is equal to their Hadamard product. 
    We prove that if the variety $X$ is not contained in a coordinate hyperplane or a binomial hypersurface, then the generic point has a finite $X$-Hadamard-rank. Although the Hadamard rank might not be well defined for special points, we prove that the general Hadamard rank with respect to secant varieties of toric varieties is finite and the maximum Hadamard rank for points with no coordinates equal to zero is at most twice the generic rank. In particular, we focus on Hadamard ranks with respect to secant varieties of toric varieties since they provide a geometric framework in which Hadamard decompositions of tensors can be interpreted. Finally, we give a lower bound to the dimension of Hadamard products of secant varieties of toric varieties: this allows us to deduce the general Hadamard rank with respect to secant varieties of several Segre-Veronese varieties. 
\end{abstract} 

\maketitle

\section{Introduction}
Restricted Boltzmann Machines are a particular type of probabilistic graphical models arising in machine learning as building blocks of deep neural networks. They can be regarded as models for \textit{tensor decompositions}, where tensors are written as \textit{Hadamard products of tensor-rank decompositions}, i.e., as entry-wise products of additive decompositions of decomposable tensors. Motivated by these structures, we investigate broader types of Hadamard product decompositions of tensors. More generally, we introduce the notion of Hadamard rank with respect to a variety and study its properties. 

In \cite{CMS10:GeometryRBM}, the authors began the study of Restricted Boltzmann Machines from the perspective of algebraic statistics and tropical geometry. 
This approach was followed by \cite{CTY10:Implicitization,montufar2015discrete,MM17:DimensionKronecker,SM18:MixturesTwoModels}, investigating, in particular, the expressive power and approximation errors of these models. An overview can be found in \cite{Mon16}. 
The size of the tensor, the number of Hadamard factors, and the length of the tensor-rank decompositions in each factor are determined by the \textit{architecture} of the Restricted Boltzmann Machine, namely the number of observable and hidden units and their number of states. 
This point of view was considered in \cite{Mon16,FOW:MinkowskiHadamard,SM18:MixturesTwoModels,oneto2023hadamard}. 
Such decompositions have been called \textit{Hadamard-Hitchcock decompositions} ($\bfr$-HHD) in \cite{oneto2023hadamard}, where $\bfr = (r_1,\ldots,r_m)$ are the lengths of the tensor-rank decompositions in each Hadamard factor. They have also been considered in the context of data analysis in \cite{ciaperoni2024hadamard}. 

In parallel to the aforementioned literature, the purely geometric notion of a \textit{Hadamard product of projective varieties} has been considered, see, e.g., \cite{BCK17,bocci2022hadamard,antolini2025algebraic,maazouz2024spinor} or the book \cite{bocci2024hadamard}. In a few words, the Hadamard product of two points in projective space is obtained from their coefficient-wise multiplication in a fixed choice of coordinates (see \Cref{eq:hadamard_product_map}). 
It is clear from the definition that this operation is not well-defined everywhere on projective spaces and that it depends on the choice of coordinates. 
The Hadamard product of two projective varieties is the Zariski closure of the set of Hadamard products of all pairs of points in the Cartesian product of the two varieties (see \Cref{def:Hadarmard-product}). 

In the classical algebraic geometry literature on additive decompositions of tensors, the notion of tensor rank has been put in the more general framework of \textit{secant varieties} and \textit{$X$-ranks}. 
Given a projective variety $X\subset \bbP^N$, the \textit{$X$-rank} of a point $p\in \bbP^N$ is the smallest number of points $x_1,\ldots,x_r\in X$ such that $p$ lies on their linear span, $p\in \langle x_1,\ldots,x_r \rangle$. Geometrically, this is related to the notion of \textit{$r$-secant variety}, i.e., the Zariski closure of the set of points of $X$-rank at most $r$. Recall that tensor rank decompositions correspond to the notion of $X$-rank with $X$ being a Segre variety (Veronese or Segre-Veronese varieties if we are interested in symmetric or partially-symmetric tensors, respectively). 
For a general overview on these definitions, see, e.g., \cite[Section 5.2.1]{Lan12:Book}. 

In this paper, we introduce a multiplicative version of the notion of $X$-rank with the goal of formalizing a geometric framework for HHDs. We define the notion of \textit{Hadamard-$X$-rank} with respect to any projective variety $X\subset \bbP^N$. Given a projective variety $X\subset \bbP^N$, the \textit{Hadamard-$X$-rank} of a point $p\in \bbP^N$ is the smallest number of points $x_1,\ldots,x_r\in X$ such that $p$ is equal to their Hadamard product (\Cref{def:hadamard-x-rank}). For example, Hadamard products of tensor rank decompositions are related to the notion of Hadamard rank with respect to Segre-Veronese varieties and to the study of Hadamard products of their secant varieties. 

Most of the paper will be focused on Hadamard ranks with respect to \textit{secant varieties} of \textit{toric varieties}. These include Segre-Veronese varieties. Also, since toric varieties are strictly related to statistical models known as \textit{exponential families}, our results on toric varieties extend some of the results of \cite{MM17:DimensionKronecker}.

\subsection*{Main results and structure of the paper}
In \Cref{sec:Hadamard_ranks}, we offer a brief overview on $X$-ranks and secant varieties and we define the notion of $X$-Hadamard-rank with respect to any algebraic variety $X\subset \bbP^N$.

In \Cref{sec:finiteness_HadamardRanks}, we study which conditions on the variety $X\subset \bbP^N$ and a point $p\in \bbP^N$ guarantee that the $X$-Hadamard-rank of $p$ is well-defined or, at least, that the general $X$-Hadamard-rank is finite. For example, in \Cref{cor:finitness_partiallysymmetric_Hadamard_ranks}, as a consequence of a more general result for toric varieties, we deduce that the Hadamard rank with respect to the $r$th secant variety ($r \geq 2$) of a Segre-Veronese variety is finite for any point in the ambient space. This can be rephrased by saying that any partially symmetric tensor admits a partially symmetric $\bfr$-Hadamard-Hitchcock decomposition for any $\bfr = (r_1,\ldots,r_m)$ with $r_i \geq 2$. We also show that the Hadamard rank of points with all coordinates different from zero is at most twice the Hadamard rank of a general point (see \Cref{prop:maximum_rank_atmost_twice_generic}). 
By means of tropical geometry, we characterize varieties $X$ for which the Hadamard-$X$-rank of a general point is finite: these are \textit{concise} varieties (i.e., not contained in a coordinate hyperplane) which are not contained in a proper binomial hypersurface or, equivalently, varieties whose ideal contains no variables and no binomials (see \Cref{cor:finite_HadamardRank_condition}). 

Finally, in \Cref{sec:dimension_Hadamard_powers}, we give a lower bound on the dimension of Hadamard products of secant varieties of a toric variety in terms of dimensions of its higher secant varieties, see \Cref{lemma:upperbound_dimensions}. In \Cref{cor:lowerbound_expfamily}, we rephrase this result in terms of exponential families thanks to their strong relations to toric varieties. This allows us to deduce the actual dimension, and then the generic $X$-Hadamard-rank, for Segre-Veronese varieties. \Cref{lemma:upperbound_dimensions} and \Cref{cor:lowerbound_expfamily} generalize previous results from \cite{MM17:DimensionKronecker}.

We believe that the present paper sets the ground for future investigations about Hadamard products of secant varieties, Hadamard ranks and Hadamard decompositions with respect to interesting families of varieties. 
We list some of them in \Cref{sec:future}. 

\subsection*{Acknowledgments}
AO acknowledges partial financial support from A.\ v.\ Humboldt Foundation/Stiftung through a fellowship for postdoctoral researchers spent at University of Magdeburg (Germany) on early stages of the collaboration with GM. 
DA and AO acknowledge the TensorDec Laboratory of the Department of Mathematics of the University of Trento, of which they are currently members, for helpful discussions. AO acknowledges that the work has been partially funded by the Italian Ministry of University and Research in the framework of the Call for Proposals for scrolling of final rankings of the PRIN 2022 call - Protocol no.\ 2022NBN7TL (\textit{``Applied Algebraic Geometry of Tensors''}). DA and AO are members of GNSAGA (INdAM). 
GM acknowledges support from the following sources: DARPA, under the AIQ project HR00112520014:P00002; 
the National Science Foundation (NSF) through awards CCF-2212520 and DMS-2145630; the German Research Foundation (DFG) through the SPP 2298 (FoDL) project 464109215; and the Federal Ministry of Research, Technology and Space (BMFTR) through the DAAD project 57616814 (SECAI). 
The authors thank Bernd Sturmfels for inspiring discussions at several stages of this work.

\section{Hadamard ranks of algebraic varieties}\label{sec:Hadamard_ranks}
\begin{notation}
    If $V$ is an $(n+1)$-dimensional $\bbC$-vector space, we denote by $\bbP V$ its projectivization. Due to the nature of Hadamard products, we will consider projective varieties with respect to a fixed choice of coordinates. In this case, we will assume a fixed basis on $V$. We will simply write $\bbP^n = \bbP V$, denoting by $(v_0:\cdots:v_n)$ the projective coordinates of the vector with coordinates $(v_0,\ldots,v_n)$ in the fixed basis.
\end{notation}

\subsection{Additive decompositions and secant varieties of algebraic varieties}
\label{ssec:secants}
In 1927, Hitchcock introduced the additive decomposition of a tensor $T\in\bbC^{n_1+1}\otimes\cdots\otimes\bbC^{n_k+1}$ as a sum of decomposable tensors, i.e., $T=\sum_{i=1}^r 
\bfv_{i,1}\otimes\cdots\otimes \bfv_{i,k}$, where $\bfv_{i,j}\in \bbC^{n_j+1}$, see \cite{hitchcock1927expression}. The smallest $r$ for which such a decomposition exists is called the \textit{tensor rank} of $T$. Even earlier, in the 1850s, Sylvester studied additive decompositions of a homogeneous polynomial $F\in {\rm Sym}^d\bbC^{n+1}$ as a sum of powers of linear forms, i.e., a decomposition as $F=\sum_{i=1}^r L_i^d$, where $L_i\in {\rm Sym}^1\bbC^{n+1}$, see \cite{sylvester1851lx}. Since homogeneous polynomials can be identified with symmetric tensors, the second decomposition is a symmetric version of the first one. The smallest $r$ for which such a decomposition exists is called the \textit{symmetric tensor rank}, or \textit{Waring rank}, of $F$. 
More generally, given a partially symmetric tensor $T\in {\rm Sym}^{d_1}\bbC^{n_1+1}\otimes\cdots\otimes{\rm Sym}^{d_k}\bbC^{n_k+1}$, we call the smallest number $r$ such that $T=\sum_{i=1}^r \bfv_{i,1}^{\otimes d_1}\otimes\cdots\otimes \bfv_{i,k}^{\otimes d_k}$ 
the \textit{partially symmetric tensor rank} of $T$.

The general geometric framework in which studying tensor decompositions and many other additive decompositions is the one of \textit{secant varieties} of projective varieties. 

\begin{definition}
\label{def:secant-variety}
    Let $X\subset \bbP^N$ be a projective variety, and let $p\in\bbP^N$. 
    The \textbf{$X$-rank} of $p$ is the smallest number of points of $X$ whose linear span contains the point $p$, i.e., 
    \[
        \rk_X(p):=\min\{r\in \bbZ_{\geq 1} ~:~ p \in \langle q_1,\ldots,q_r\rangle, \ q_i\in X\}.
    \]
    The \textbf{$r$th secant variety} of $X$ is the Zariski closure of the set of points of $X$-rank at most $r$, i.e., 
    \[
        \sigma_r(X) := \overline{\{p \in \bbP^N ~:~ \rk_X(p)\leq r\}} = \overline{\bigcup_{q_1,\ldots,q_r\in X}\langle q_1,\ldots,q_r\rangle}.
    \]
\end{definition}

The aforementioned additive decompositions of tensors correspond to the notion of $X$-rank with respect to Segre, Veronese and Segre-Veronese varieties, respectively. Given a partition ${\bf d} = (d_1,\ldots,d_k)$ of $d\in \bbZ_{\geq 1}$ and ${\bf n} = (n_1, \ldots, n_k) \in \bbZ_{\geq 1}^k$, the \textbf{Segre-Veronese variety} $SV_{\bf d, n}$ of rank-one partially symmetric tensors is the image of the map
\begin{equation}\label{eq:segreveronese_parametrization}
\begin{matrix}
    \nu_{\bf d} & : & \bbP\bbC^{n_1+1}\times\cdots\times\bbP\bbC^{n_k+1} & \rightarrow & \bbP({\rm Sym}^{d_1}\bbC^{n_1+1}\otimes\cdots\otimes{\rm Sym}^{d_{k}
    }\bbC^{n_k+1}), \\
    & & ([\bfv_1],\ldots,[\bfv_k]) & \mapsto & [\bfv_1^{\otimes d_1}\otimes\cdots\otimes \bfv_k^{\otimes d_k}].
\end{matrix}
\end{equation}

The standard representation in coordinates is the monomial embedding given by 
\begin{equation}\label{eq:segreveronese_parametrization_variables}
\begin{matrix}
    \nu_{\bf d} & : & \bbP^{n_1}\times\cdots\times\bbP^{n_k} & \rightarrow & \bbP^N, & \qquad N = \prod_i {n_i+d_i \choose d_i} - 1\\
    & & ({\bf x}_1,\ldots,{\bf x}_k) & \mapsto &  (\cdots: {\bf x}_1^{\alpha_1}\cdots{\bf x}_k^{\alpha_k}:\cdots)_{\substack{\alpha_i \in \bbZ_{\geq 0}^{n_i+1} \\ |\alpha_i| = d_i}} & 
\end{matrix}
\end{equation}
where ${\bf x}_i = (x_{i,0}:\cdots:x_{i,n_i})\in\bbP^{n_i}$, $\alpha_i = (\alpha_{i,0},\ldots,\alpha_{i,n})\in\bbZ_{\geq 0}^{n_i+1}$, and ${\bf x}_i^{\alpha_i} := \prod_{j=0}^{n_i} x_{i,j}^{\alpha_{i,j}}$. 

The case ${\bf d}=\mathds{1} := (1,\ldots,1)$ corresponds to {\it Segre varieties} $S_{\bf n} := SV_{\mathds{1},\bfn}$. 
The case $k=1$, 
i.e., ${\bf d}=(d)$, corresponds to {\it Veronese varieties} $V_{d,n} := SV_{(d), (n)}$.

An extensive amount of literature has been dedicated to the study of $X$-ranks of these and many other projective varieties. We refer to \cite{BCCGO:Hitchhiker} for a general overview on these topics. 

\subsection{Multiplicative decompositions and Hadamard powers of algebraic varieties} 
\label{ssec:hadamard_powers}

In \cite{CMS10:GeometryRBM}, Cueto, Morton and Sturmfels developed an algebraic geometry approach to study a statistical model known as the \textit{Restricted Boltzmann Machine} (RBM). The probability distributions arising from such a model can be interpreted as Hadamard products of tensors of prescribed tensor rank, see \cite{Mon16}. Such multiplicative decompositions of tensors have been considered also in \cite{FOW:MinkowskiHadamard,SM18:MixturesTwoModels,oneto2023hadamard}. 

The geometric framework for such multiplicative decompositions is given by \textit{Hadamard products} of algebraic varieties. Once a choice of coordinates is fixed, the coefficient-wise product, known as \textit{Hadamard product}, defines the polynomial map
\begin{equation}\label{eq:hadamard_product_map}
\begin{matrix}
    h & : & \bbC^{N+1} \times \bbC^{N+1} & \rightarrow & \bbC^{N+1}, \\
    & & (\bfx,\bfy) = ((x_0,\ldots,x_N),(y_0,\ldots,y_N)) & \mapsto & \bfx \star \bfy := (x_0y_0,\ldots,x_Ny_N).
\end{matrix}
\end{equation}
This extends to a rational map of projective spaces $h \colon \bbP^N \times \bbP^N \dashrightarrow \bbP^N$. Clearly, the latter map is not defined everywhere, e.g., we cannot Hadamard-multiply two distinct projective coordinate points. Also, as already mentioned, the map depends on the choice of a basis. However, in many applications, a natural choice of basis is forced upon us. This happens, for instance, when we consider a tensor representing the joint probabilities of a set of discrete random variables, as in our motivating example of RBMs. 

\begin{notation}
    We denote by $\mathds{1}$ both the vector with all entries equal to one and the corresponding projective point $(1:\cdots:1) \in \bbP^N$. Note that it is the identity of the Hadamard product. If all coordinates of $p \in \bbP^N$ are non-zero, we write $p^{\star (-1)}$ for its Hadamard-inverse, i.e., $p\star p^{\star (-1)} = \mathds{1}$.
\end{notation}

\begin{definition}\label{def:hadamard-x-rank}
    Let $X\subset \bbP^N$ be a projective variety, and let $p\in \bbP^N$. 
    The \textbf{$X$-Hadamard-rank} of $p$ is the smallest number of points on $X$ whose Hadamard product is $p$, i.e., 
    \[
        \Hrk_{X}(p) := \min\{m \in \bbZ_{\geq 1} ~:~ p = q_1\star\cdots\star q_m, \ q_i \in X\} , 
    \]
    or $\Hrk_{X}(p) := \infty$ if $p$ cannot be expressed as a Hadamard product of points in $X$. 

    We write 
    \[
        \eta_m(X) := \overline{\{p \in \bbP^n~:~ \Hrk_{X}(p)\leq m\}}
    \]
    to indicate the Zariski closure of the set of points having $X$-Hadamard-rank at most equal to $m$. 

    We denote by $\Hrk_{X}^\circ$ the \textbf{generic $X$-Hadamard-rank}. That is the smallest $m \in \bbZ_{\geq 1}$ such that $\eta_m(X)=\bbP^{N}
    $. If such an $m$ does not exist, we write $\Hrk_X^\circ := \infty$.
\end{definition}

The definition of $\eta_m(X)$ is related to the notion of {\it Hadamard products} of projective varieties introduced in \cite{CMS10:GeometryRBM}. We refer to \cite{bocci2024hadamard} for a recent monograph on the subject and for more references.

\begin{definition}
\label{def:Hadarmard-product}
    Let $X,Y\subset \bbP^N$ be projective varieties. The \textbf{Hadamard product} of $X$ and $Y$ is
    \begin{equation}\label{def:HadamardProduct}
        X \star Y := \overline{h(X\times Y)} = \overline{\{p \star q ~:~ p\in X, \ q \in Y, \ p \star q \text{ exists} \}}.
    \end{equation}
    The \textbf{Hadamard powers} of $X$ are defined recursively by
    \[
        X^{\star 0} := \mathds{1} \quad \text{ and }\quad X^{\star m} := X \star X^{\star (m-1)}.
    \]
    Equivalently,
    \[
        X^{\star m} = \overline{\{ q_1 \star \cdots \star q_m : q_1, \ldots, q_m \in X, \ q_1 \star \cdots \star q_m \text{ exists} \}}.
    \]
\end{definition}

As already mentioned, the Hadamard product of algebraic varieties is highly dependent on the choice of coordinates. If $g \in PGL(N)$, then $X\star Y$ can differ from $gX\star gY$. For example, the Hadamard product of a general pair of points is well-defined; however, the Hadamard product of two different coordinate points, such as $(1:0)$ and $(0:1)$ in $\bbP^1$, is not defined. See \cite{BALLICO} for Hadamard products of varieties after a generic change of coordinates. Nevertheless, the Hadamard product is well-behaved with respect to the action of the torus of diagonal matrices; see \Cref{lemma:diagonal_iso} and its consequences.

\begin{remark}
	The variety $X \star Y \subset \bbP V$ can be seen also as the linear projection of the Segre product $X \times Y \subset \bbP(V \otimes V)$ on the diagonal coordinates. By continuity, if $X$ and $Y$ are both irreducible, then $X \star Y$ and all powers $X^{\star m}$ are irreducible.
\end{remark}

\begin{example}
	It is immediate to see that the $X$-Hadamard-rank of a point can be infinite. For example, since an \textit{embedded toric variety} (see \Cref{def:toricvarieties}) $X \subset \bbP^N$ is defined by monomial embeddings and by binomial equations, it is \textit{Hadamard-idempotent}, i.e., $X^{\star 2} = X$, see e.g., \cite[Proposition 4.7]{FOW:MinkowskiHadamard}. In particular, the $X$-Hadamard-rank of any point $p \in \bbP^N$ is either equal to one (if $p\in X$) or it is infinite. 
    
    We will see more on the importance of Hadamard-idempotent varieties in the study of the finiteness of $X$-Hadamard-ranks. See \Cref{sec:finiteness_HadamardRanks}. 
\end{example} 

\begin{remark} 
    We will always assume that $X\subset \bbP^N$ is \textbf{concise}, i.e., it is not contained in any coordinate hyperplane. Indeed, since coordinate hyperplanes are Hadamard-idempotent, if $X$ is not concise, then also all its Hadamard powers are not concise. In such a case, we regard $X$ in its concise ambient space by forgetting the coordinates of the hyperplanes containing it. Note that this is not the case for degenerate varieties contained in arbitrary linear subspaces: indeed, their Hadamard powers might be no longer degenerate. 
\end{remark}

Clearly $X^{\star m} \subset \eta_m(X)$, but it is not difficult to see that, in general, the inclusion might be strict. 

\begin{example}\label{example:powers_vs_ranks}
	Let $p = (1:2) \in \bbP^1$ and $X = \{p\}$. Then, $X^{\star m} = \{(1:2^m)\}$ while $\eta_m(X) = \{(1:2^r)~:~r \leq m\}$. 
\end{example}

\Cref{example:powers_vs_ranks} immediately highlights where the issue is. We always have that 
$
	\eta_m(X) = \bigcup_{r=1}^m X^{\star r} , 
$
but a Hadamard power of $X$ might not be contained in the following one. However, for concise varieties, we have that $\dim \eta_m(X) = \dim X^{\star m}$, since $\dim X^{\star s} \leq \dim X^{\star (s+1)}$ for all $s \geq 1$: indeed, for every $p\in X$ with all coordinates different from zero, $p\star X^{\star s}$ is isomorphic to $X^{\star s}$ and $p\star X^{\star s} \subset X^{\star (s+1)}$. Even if $X$ is irreducible, $\eta_m(X)$ may be a reducible variety and strictly contain $X^{\star m}$. An assumption that guarantees the inclusion between consecutive Hadamard powers and, therefore, the equality $\eta_m(X) = X^{\star m}$, is that $\mathds{1}\in X$. Under this assumption, 
\[
	X\subset X^{\star 2}\subset \cdots \subset X^{\star s} \subset \cdots \subset \bbP^N \quad \text{ and } \quad \eta_s(X) = X^{\star s} \text{ for any }s.
\]
The viceversa is not always true, as shown in the following example. 
\begin{example}
    Consider the following points in $\bbP^3$:
    \[
        p = (1:1:2:3), \quad q = p^{\star(-1)} = (6:6:3:2) . 
    \]
    Let $L = \langle p,q\rangle$ be the line through $p$ and $q$. This has equation: 
\[
L\colon \qquad \begin{cases}
    5x_0 - 16x_2 + 9x_3 = 0\\
    5x_1 - 16x_2 + 9x_3 = 0
\end{cases}.
\]
In \cite[Remark 3.1.1]{antolini2025algebraic} it is shown that in this case $L^{\star 2} = \langle p^{\star 2},p\star q, q^{\star 2}\rangle$. This is a toric plane in $\bbP^3$:
\[
L^{\star 2}\colon \qquad x_0 - x_1 = 0.
\]
Since $L^{\star 2}$ is linear and $p, q \in L^{\star 2}$, we have that $L \subset L^{\star 2} \subset L^{\star 3} = L^{\star 2} \subsetneq \bbP^3$ even if $\mathds{1}\not \in L$.
\end{example}

\begin{example}
	If we consider the $k$-factor {\it Segre product} $S_{\mathds{1}}$, then $\sigma_2(S_{\mathds{1}})^{\star m}$ is the algebraic variety corresponding to the {\it Restricted Boltzmann Machine} with $m$ binary hidden units and $k$ binary observed units, see \cite{Mon16}. In \cite{montufar2015discrete}, the authors introduced {\it Discrete Restricted Boltzmann Machines} by allowing non-binary units. From our definitions, this corresponds to Hadamard powers of secant varieties of Segre varieties. With our more geometric interpretation, we might extend the notion of Discrete RBM models to the case of symmetric tensors, partially symmetric tensors or skew-symmetric tensors by looking at Hadamard products of secant varieties of Veronese varieties, Segre-Veronese varieties and Grassmannians, respectively.
\end{example}

Having the latter as guiding example, the general framework we consider is the notion of Hadamard rank with respect to secant varieties. Hence, we generalize the definition of $X$-Hadamard-rank as follows. 

\begin{definition}
    Let $X\subset \bbP^N$ be a projective variety, and let $p\in \bbP^N$. 
    The \textbf{$r$th $X$-Hadamard-rank} of $p$ is the smallest number of points on $\sigma_r(X)$ such that their Hadamard product is equal to $p$, i.e., 
    \[
        \Hrk_{X,r}(p) := \min\{m \in \bbZ_{\geq 1} ~:~ p = q_1\star\cdots\star q_m, \ q_i \in \sigma_r(X)\} , 
    \]
    or $\Hrk_{X,r}(p) := \infty$ if $p$ cannot be expressed as a Hadamard product of points in $\sigma_r(X)$. 

    Analogously, we denote by $\Hrk_{X,r}^\circ$ the \textbf{generic $r$th $X$-Hadamard-rank}. That is the smallest $m$ such that $\sigma_r(X)^{\star m}$ is equal to $\bbP^N$. 
    
    Note per \Cref{def:secant-variety} and \Cref{def:Hadarmard-product} that $\sigma_r(X)$ and $\sigma_r(X)^{\star m}$ are defined as Zariski closures. In particular, there may exist points $p$ whose $r$th $X$-Hadamard-rank $\Hrk_{X,r}(p)$ is larger than the generic $r$th $X$-Hadamard-rank $\Hrk_{X,r}^\circ$. Actually, in \Cref{example:maximum_nonfinite}, we will see a case in which the generic $X$-Hadamard-rank is finite but there exists an entire locus of points in which the $X$-Hadamard-rank is not even well-defined.
    
    Clearly, the $X$-Hadamard-rank in \Cref{def:hadamard-x-rank} correspond to the case $r=1$.
\end{definition}

\begin{remark}\label{rmk:Hadamard_rank}
    Let $SV_{\bf d, n}$ be a Segre-Veronese variety, i.e., the variety of partially symmetric rank-one tensors in $\bbP({\rm Sym}^{d_1}\bbC^{n_1+1}\otimes\cdots\otimes{\rm Sym}^{d_k}\bbC^{n_k+1})$. The notion of $r$th $X$-Hadamard-rank can be rephrased in terms of decompositions of tensors. If $T\in {\rm Sym}^{d_1}\bbC^{n_1+1}\otimes\cdots\otimes{\rm Sym}^{d_k}\bbC^{n_k+1}$, its $r$th $X$-Hadamard-rank, that we denote by $\Hrk_r^{\mathbf{d}}(T)$, consists of the smallest number $m$ of tensors $A_1,\ldots,A_m$ of partially symmetric rank at most $r$ such that $T = A_1\star \cdots \star A_m$. This general case reduces to the case of decompositions of homogeneous polynomials as Hadamard products of Waring decompositions in the case of Veronese varieties ($\bfd = (d))$ or to the case Hadamard products of tensor-rank decompositions in the case of Segre varieties ($\bfd = \mathds{1})$. 
\end{remark}

Inspired by the literature on decompositions of tensors, it is immediate to raise the following questions.

\begin{question}\label{question:A}
	Under which conditions on $X \subset \bbP^N$ and $p \in \bbP^N$ is the $r$th $X$-Hadamard-rank of $p$ finite? 
    Under which conditions on $X$ is the generic $r$th $X$-Hadamard-rank finite, i.e., $\sigma_r(X)^{\star m} = \bbP^N$ for some $m$?
\end{question}

\begin{question}\label{question:B}
	How can we compute the $r$th $X$-Hadamard-rank of a given point $p \in \bbP^N$?
    Can we find bounds on the maximal $r$th $X$-Hadamard-rank and bounds on the generic $r$th $X$-Hadamard-rank? 
\end{question}

\section{On the finiteness of Hadamard ranks}\label{sec:finiteness_HadamardRanks}
In this section, we focus on \Cref{question:A}. Before approaching it in general, we give an immediate and straightforward proof of the finiteness of $r$th Hadamard ranks with respect to Segre varieties for $r \geq 2$, i.e., finiteness of decompositions of tensors as Hadamard products of rank-$r$ tensors. These are the decompositions of tensors corresponding to the Discrete Restricted Boltzmann Machines \cite{montufar2015discrete} and called $r$-HHDs in \cite{oneto2023hadamard}.

\begin{example}\label{ex:finitness_Hadamard_ranks}(Finiteness of Hadamard ranks of tensors)
    Let $T \in\bbC^{n_1+1}\otimes\cdots\otimes\bbC^{n_k+1}$. 
    The Hadamard rank of the projective point $[T]$ with respect to the Segre variety $S_{{\bfn}} \subset \bbP (\bbC^{n_1+1}\otimes\cdots\otimes\bbC^{n_k+1})$ is equal to the smallest number of factors of an expression $T=A_1\star\cdots\star A_m$, where $\rk A_i\leq r$. 
    Extending the idea of \cite[Proposition 4.13]{FOW:MinkowskiHadamard}, it is immediate to see that $T$ is always the Hadamard product of finitely many tensors of rank at most $r$, for any $r \geq 2$. Let $T = (t_{i_1,\ldots,i_k})_{i_1,\ldots,i_k}$ be a presentation in the chosen basis of $\bbC^{n_1+1}\otimes\cdots\otimes\bbC^{n_k+1}$ and let $\bfv_{i_2,\ldots,i_k} := (t_{0,i_2,\ldots,i_k},\ldots,t_{n_1,i_2,\ldots,i_k}) \in \bbC^{n_1+1}$. Then, 
	\[
		T = \bigstar_{i_2,\ldots,i_k} \left[(\bfv_{i_2,\ldots,i_k}-\mathds{1})\otimes {\bf e}_{i_2}\otimes\cdots\otimes{\bf e}_{i_k} + \mathds{1}\otimes\cdots\otimes\mathds{1}\right].
	\]
    Clearly, this is far from being an optimal expression.
\end{example}

\begin{remark}
\label{rmk:RBM}
	The finiteness of generic $r$th Hadamard ranks with respect to the Segre variety for $r \geq 2$, i.e., the fact that the $m$th Hadamard power of the $r$th secant variety of a Segre variety fills the ambient space for some $m$, follows also from \cite[Theorem 13]{montufar2015discrete}. Here, the authors give a bound on such an $m$ with respect to the size of the tensor and $r$. They do it by showing that Restricted Boltzmann Machines are {\it universal approximators}, namely that they can approximate arbitrary well any joint probability distribution on the discrete sample space $\calX_1 \times \cdots \times \calX_d$ with $\calX_i =\{ 1, \ldots, n_i\}$. In other words, the model defines a dense subset of the probability simplex with respect to the Euclidean topology. In algebraic geometric terms, this implies that the model is not contained in any proper algebraic subvariety and, therefore, its Zariski closure fills the ambient space. Recall that being a Zariski dense subset over the real numbers is weaker than being Euclidean dense. In \Cref{cor:finitness_partiallysymmetric_Hadamard_ranks}, we extend this result to the case of any partially-symmetric tensor as a corollary of a geometric result on generic $r$th $X$-Hadamard-ranks with respect to embedded toric varieties.
\end{remark}

\subsection{On the finiteness of maximum $X$-Hadamard-ranks}\label{ssec:finiteness_Hadamard_maximum}

We approach \Cref{question:A} by presenting first some particular geometric property on the projective variety $X\subset \bbP^N$ which guarantees the finiteness of $X$-Hadamard-rank for \textit{any} point in $\bbP^N$. 

In \cite{BCK17}, Hadamard products of linear spaces have been systematically studied. In particular, they obtain the following useful results about Hadamard powers of lines.

\begin{notation}
    Let $\Delta_i \subset \bbP^N$ be the set of points having at most $i+1$ non-zero coordinates in the fixed basis.
\end{notation}

\begin{lemma}\label{lemmas:BCK17}
	Let $L \subset \bbP^N$ be a line such that $L \cap \Delta_{N-2} = \emptyset$. Then:
	\begin{enumerate}
		\item \cite[Lemma 2.10]{BCK17} $L^{\star s} = \bigcup_{q_1,\ldots,q_s \in L} (q_1\star\cdots\star q_s)$, i.e., the closure in \eqref{def:HadamardProduct} is not necessary;
		\item \cite[Theorem 3.4]{BCK17} $L^{\star s}$ is a projective linear space of dimension equal to $\min\{s,N\}$.
	\end{enumerate}
\end{lemma}
\Cref{lemmas:BCK17} allows us to deduce the finiteness of the $X$-Hadamard-rank of any point in the ambient space of $X$ under the assumption that $X$ contains a line $L$ such that $L \cap \Delta_{N-2}=\emptyset$.

\begin{proposition}
\label{cor:finiteness_Hadamard_ranks}
	If $L \subset X \subset \bbP^N$, where $X$ is any projective variety and $L$ is a line such that $L \cap \Delta_{N-2} = \emptyset$, then the $X$-Hadamard-rank is finite for any point in $\bbP^N$.
\end{proposition}
\begin{proof}
	Since $L \subset X$, then $\Hrk_X(p) \leq \Hrk_L(p)$ for every $p \in \bbP^N$. By \Cref{lemmas:BCK17}, we know that $L^{\star N} = \{p ~:~ \Hrk_L(p) \leq N\} = \bbP^N$. Hence, $\Hrk_L(p) \leq N$ for any $p \in \bbP^N$.
\end{proof}

As an immediate consequence, we extend \Cref{ex:finitness_Hadamard_ranks} to $r$th $X$-Hadamard-ranks with respect to \textit{toric varieties}. We recall here the definition of toric variety that we will use through the paper: this coincides with what in the algebraic geometry literature is often referred to as an \textit{embedded} toric variety, in contrast to the more general definition of \textit{abstract} toric varieties, see \cite[Section 2.3]{cox2011toric}.

\begin{notation}
    We write $\bbC^\times := \bbC \smallsetminus \{0\}$. \\
    Given a vector $\alpha \in \bbZ^{n+1}_{\geq 0}$ and a vector $\bfx \in \bbC^{n+1}$, we use the notation for monomials $\bfx^\alpha := x_0^{\alpha_0}\cdots x_n^{\alpha_n}.$
\end{notation}

\begin{definition}
\label{def:toricvarieties}
    An \textbf{(embedded) toric variety} is a positive-dimensional projective variety $X\subset \bbP^N$ defined as the Zariski closure of the image of a monomial map $\varphi_A \colon (\mathbb{C}^\times)^{n+1} \to \bbP^N, \ \bfx \mapsto (\bfx^{\alpha_0}:\cdots:\bfx^{\alpha_N})$, where $A = (\alpha_0|\cdots|\alpha_N) \in\bbZ_{\geq 0}^{(n+1) \times (N+1)}$ is a matrix with $|\alpha_i|=|\alpha_j|$ for any $i \neq j$.
    We will always assume that toric varieties are non-degenerate in the sense that $\alpha_i \neq \alpha_j$ for all $i\neq j$. We denote by $\widehat{\varphi}_A\colon (\bbC^\times)^{n+1} \to \bbC^{N+1}$ the corresponding affine monomial map when we interpret it as parametrizing the affine cone of $X$ in $\bbC^{N+1}$.
\end{definition}

\begin{remark}\label{rmk:reparametrize_toric}
    An embedded toric variety as in \Cref{def:toricvarieties} is determined by a matrix $A \in \bbZ^{(n+1) \times (N+1)}$. Recall that such an embedded toric variety $X \subset \bbP^N$ is uniquely determined by the row span of the matrix $A$: indeed, the ideal of the toric variety is uniquely determined by the kernel of the matrix $A$, see, e.g., \cite[Proposition 1.1.9]{cox2011toric}. Moreover, since we are considering projective toric varieties given by a homogeneous monomial parametrization, we have that $\mathds{1}\in \text{rowspan}(A)$. 
    In this way,  given a toric variety as in \Cref{def:toricvarieties}, we will often consider a monomial map induced by a matrix $\bar{A} \in \bbZ^{(n+1) \times (N+1)}$ having the same rowspan of $A$, but with the first row equal to $\mathds{1}$ and the first column equal to $(1,0,\ldots,0)$; in this way we get a monomial map $\varphi_{\bar{A}} \colon (\bbC^\times)^{n+1} \rightarrow U_0$ where $U_0$ is the affine chart of $\bbP^N$ with the first coordinate different from $0$. The images of the two monomial maps $\varphi_A$ and $\varphi_{\bar{A}}$ might be different, but they have the same Zariski closure in $\bbP^N$.
\end{remark}
\begin{notation}
    We denote $[n] := \{0,1,\ldots,n\}$.
\end{notation}

\begin{example}
\label{example:reparametrize_segre}
    Consider the Segre variety $S_{\bfn}$ whose usual embedding is given by the map $\nu_{\mathds{1}}$ defined in \eqref{eq:segreveronese_parametrization_variables} for $d_1 = \cdots = d_k = 1$. This is the monomial parametrization corresponding to the matrix $B_{\bf n}$ whose rows are labeled by $\{(j,i_j) ~:~ j \in \{1,\ldots,k\}, i_j \in [n_k]\}$, columns are labeled by multi-indices $\mathbf{i}' = (i'_1,\ldots,i'_k) \in [n_1]\times\cdots\times [n_k]$, and $[B_{\bf n}]_{(j,i_j),\mathbf{i}'} = \delta_{i_j,i'_j}$. For example, the matrix defining the Segre embedding of $\bbP^1 \times \bbP^2$ in $\bbP^5$ is given by the matrix
    \[
        B_{1,2} = \begin{pNiceMatrix}[first-col,first-row]
              & (0,0) & (0,1)  & (0,2) & (1,0) & (1,1) & (1,2) \\
            (1,0) & 1 & 1 & 1 & 0 & 0 & 0 \\
            (1,1) & 0 & 0 & 0 & 1 & 1 & 1 \\
            (2,0) & 1 & 0 & 0 & 1 & 0 & 0 \\
            (2,1) & 0 & 1 & 0 & 0 & 1 & 0 \\
            (2,2) & 0 & 0 & 1 & 0 & 0 & 1 \\
\end{pNiceMatrix},
    \]
    corresponding to the monomial map 
    \[
    \begin{array}{c c c c}
        \varphi_{B_{1,2}} : & (\bbC^\times)^2 \times (\bbC^\times)^3 & \longrightarrow & \bbP^5 \\
        & ((a_{1,0},a_{1,1}),(a_{2,0},a_{2,1},a_{2,2})) & \longmapsto &(a_{1,0}a_{2,0}:a_{1,0}a_{2,1}:a_{1,0}a_{2,2}:a_{1,1}a_{2,0}:a_{1,1}a_{2,1}:a_{1,1}a_{2,2}).
    \end{array}
    \]
    However, such a matrix $B_{\bfn}$ is not full rank. In particular, we can consider a monomial parametrization associated to the matrix $\bar{B}_{\bfn}$ obtained by deleting all the rows corresponding to $(i,0)$, for all $i = 1,\ldots,k$, and adding a first row equal to $\mathds{1}$, labeled by $0$; namely,
    \[
        \bar{B}_{1,2} = \begin{pNiceMatrix}[first-col,first-row]
              & (0,0) & (0,1)  & (0,2) & (1,0) & (1,1) & (1,2) \\
            0 & 1 & 1 & 1 & 1 & 1 & 1 \\
            (1,1) & 0 & 0 & 0 & 1 & 1 & 1 \\
            (2,1) & 0 & 1 & 0 & 0 & 1 & 0 \\
            (2,2) & 0 & 0 & 1 & 0 & 0 & 1 \\
        \end{pNiceMatrix},
    \]
    corresponding to the monomial map
    \[
    \begin{array}{c c c c}
        \varphi_{\bar{B}_{1,2}} : & (\bbC^\times)^4 & \longrightarrow & \bbP^5 \\
        & (a_0,a_{1,1},a_{2,1},a_{2,2}) & \longmapsto & (a_{0}:a_{0}a_{2,1}:a_{0}a_{2,2}:a_{0}a_{1,1}:a_0a_{1,1}a_{2,1}:a_0a_{1,1}a_{2,2}).
    \end{array}
    \]
    Note that the map $\varphi_{B_{1,2}}$ is surjective onto the Segre variety $S_{1,2}$, while the image of $\varphi_{\bar{B}_{1,2}}$ is equal to the affine chart of $S_{1,2}$ given by the subset of points with first coordinate different than zero.
\end{example}  

Note that if $X \subset \bbP^N$ is an embedded toric variety, then $\mathds{1} =\varphi_A(\mathds{1})\in X$ and the generic point $\bfx = (x_0:\cdots:x_N) \in X$ has pairwise distinct entries, i.e., $\det \left(\begin{smallmatrix}
	1 & 1 \\ x_i & x_j
\end{smallmatrix}\right) \neq 0$ for all $i \neq j$. Therefore, for a generic $\bfx \in X$, $L = \langle \mathds{1}, \bfx\rangle$ is a line that does not contain any point with two zero-entries, i.e., it satisfies the assumptions of \Cref{lemmas:BCK17} for $\bfx \in X$ generic. Such a line is a secant line to the variety $X$ and, in particular, $L \subset \sigma_r(X)$ for any $r \geq 2$. Therefore, from this observation, together with \Cref{cor:finiteness_Hadamard_ranks}, we deduce the following. 
\begin{corollary}\label{cor:finiteness_Hadamard_ranks_toric}
	Let $X$ be a non-degenerate toric variety $X\subset \bbP^N$. Then, for any $r \geq 2$, the $r$th $X$-Hadamard-rank is finite for any point in $\bbP^N$.
\end{corollary}
Since Segre-Veronese varieties are toric, we extend \Cref{ex:finitness_Hadamard_ranks} to partially-symmetric tensors. 

\begin{corollary}[Finiteness of Hadamard ranks of partially-symmetric tensors]\label{cor:finitness_partiallysymmetric_Hadamard_ranks}
	Fix any $r \geq 2$. Then, any partially-symmetric tensor $T \in \Sym^{d_1}\bbC^{n_1+1} \otimes \cdots \otimes \Sym^{d_k}\bbC^{n_k+1}$ can be written as a Hadamard product of finitely many partially-symmetric tensors of partially-symmetric rank at most $r$. 
\end{corollary}

Even if \Cref{cor:finiteness_Hadamard_ranks_toric} gives a positive answer to \Cref{question:A} for interesting varieties related to tensor decompositions, it is restrictive as it requires that the variety contains a special line as in \Cref{cor:finiteness_Hadamard_ranks}.

In order to relax the problem, in \Cref{sec:infiniteness_generic_Hrk} we study finiteness of \textit{generic} $X$-Hadamard-ranks. In the process of relaxing the problem, it is natural to ask whether the finiteness of the generic $X$-Hadamard-rank would be enough to guarantee the finiteness of the $X$-Hadamard-rank of \textit{any} point in the ambient space. The following example shows that, for very special points, this is not the case. 

\begin{example}\label{example:maximum_nonfinite}
    If $X\subset \bbP^N$ and the generic $X$-Hadamard-rank is finite, then there exists a Zariski open set $U \subset \bbP^N$ such that $\Hrk_X(p) \leq \Hrk_X^\circ$ for all $p \in U$. However, this does not imply that $\Hrk_X(p) < \infty$ for all $p \in \bbP^N$. As an example, let $Q = \{ x_0x_1 + x_0x_2 + x_1x_2 = 0 \} \subset \bbP^2$. Then, $Q^{\star 2} = \bbP^2$ and $\Hrk_Q^\circ = 2$. However, every point $(a:b:0)$ such that $ab \neq 0$ cannot be written as a product of two points on $Q$. In fact, suppose that $(a:b:0) = (x_0y_0:x_1y_1:x_2y_2)$ for some $(x_0:x_1:x_2), (y_0:y_1:y_2) \in Q$. Then, $x_2y_2 = 0$. Suppose that $x_2 = 0$. Then, since $(x_0:x_1:x_2) \in Q$, $x_0x_1 = 0$, from which $ab= 0$, which is a contradiction. Same if $y_2 = 0$. In particular, for all $a, b \neq 0$, we have that $\Hrk_Q((a:b:0)) = \infty$.
\end{example}

However, the type of failure that we have illustrated in the latter example cannot occur if we consider points with all coordinates different from zero. 

\begin{proposition}\label{prop:maximum_rank_atmost_twice_generic}
    Let $X\subset \bbP^N$ be an irreducible variety with $\Hrk_X^\circ<\infty$ and $\bfx = (x_0:\cdots:x_N) \in \bbP^N$ be any point such that $\prod_i x_i \neq 0$. Then, $\Hrk_X(\bfx)\leq 2 \cdot \Hrk_X^\circ$.
\end{proposition}
\begin{proof}
    Since $\Hrk_X^\circ < \infty$, there exists a non-empty Zariski open set $U \subset \bbP^N$ such that $\Hrk_X(\bfy)\leq\Hrk_X^\circ$ for all $\bfy \in U$. Since, by assumption, $\prod_i x_i \neq 0$, the coordinate-wise multiplication map $m_\bfx \colon \bbP^N \to \bbP^N, \ \bfy \mapsto \bfx\star \bfy$ is an isomorphism. Let $\text{Crem}_N \colon \bbP^N \dashrightarrow \bbP^N, \ \bfy\mapsto \bfy^{\star (-1)}$ be the Cremona map. Since $\text{Crem}_N$ is birational and $m_\bfx$ is an isomorphism, both $\text{Crem}_N^{-1}(U)$ and $m_\bfx^{-1}(U)$ are non-empty Zariski open subsets. Hence, $m_\bfx^{-1}(U)\cap \text{Crem}_N^{-1}(U)$ is a non-empty Zariski open set. For any $\bfy \in m_\bfx^{-1}(U)\cap \text{Crem}_N^{-1}(U)$, we can write $\bfx = (\bfx\star \bfy)\star \bfy^{\star (-1)}$. Since $\bfx \star \bfy, \ \bfy^{\star (-1)} \in U$, the claim follows.
\end{proof}

Interestingly, the previous result can be regarded as a \textit{multiplicative} version of the fact that the maximum $X$-rank is at most twice the generic $X$-rank, see \cite[Theorem 1]{blekherman2015maximum}.   

\subsection{On the (in)finiteness of the generic $X$-Hadamard-rank}\label{sec:infiniteness_generic_Hrk} Here, we consider the relaxed version of \Cref{question:A} by focusing on finiteness of \textit{generic} $X$-Hadamard-ranks under the only assumption that $\mathds{1}\in X$. As already mentioned, this implies that
\begin{equation}\label{eq:chain_inclusions}
	X \subset X^{\star 2} \subset \cdots \subset X^{\star r} \subset \cdots \subset \bbP^N.
\end{equation}
Recall that if $X$ is an irreducible variety, then $X^{\star i}$ is an irreducible variety for all $i \in \bbZ_{\geq 1}$. In particular, we have that the chain of Hadamard powers of $X$ in \eqref{eq:chain_inclusions} eventually stabilizes. 
We are interested in understanding under which conditions it stabilizes \textit{before} filling the ambient space.

A necessary condition for the finiteness of the generic $X$-Hadamard-rank is that $X$ is not contained in a proper \textit{Hadamard-idempotent variety}, i.e., a variety $Y$ such that $Y^{\star 2} = Y \neq \mathbb{P}^N$. Indeed, if $X\subset Y\subsetneq \bbP^N$ with $Y^{\star 2} = Y$, then $X^{\star m} \subset Y^{\star m} = Y\subsetneq \bbP^N$. For example, since they are defined by monomial embeddings, embedded toric varieties are Hadamard-idempotent, see also \cite[Proposition 4.7]{FOW:MinkowskiHadamard}. 

We give a complete characterization of projective varieties whose generic $X$-Hadamard-rank is infinite: these are the projective varieties contained in Hadamard-idempotent varieties \textit{up to a diagonal coordinate change}. 

\begin{notation}
    We denote by $\mathbb{T} \subset PGL(N)$ the torus of complex diagonal $(N+1)\times(N+1)$ matrices with non-zero entries on the diagonal, up to non-zero scaling.
\end{notation}

\begin{proposition}
\label{prop:contained_in_idempotent}
	Let $X \subset \bbP^N$ be an irreducible variety. If the generic $X$-Hadamard-rank is infinite, then there exists a diagonal invertible matrix $t\in \mathbb{T}$ such that $tX$ is contained in a proper Hadamard-idempotent irreducible variety.
\end{proposition}

In order to prove \Cref{prop:contained_in_idempotent}, we need some preliminary results. 

\begin{lemma}
\label{lemma:diagonal_iso}
	Let $X, Y \subset \bbP^N$ be irreducible varieties and let $t, t' \in \mathbb{T} \subset PGL(N)$ be diagonal invertible matrices. Then, $t X \star t' Y = tt' (X \star Y)$.
\end{lemma}

\begin{proof}
    Since $t,t'\in \bbT$ are diagonal, then, for any $p,q\in\bbP^N$, we have $tt'(p\star q)=(tp)\star(t'q)$.
	This shows that $tt' (X \star Y) \subset t X \star t' Y$. Moreover, since $t$, $t'$ and $tt'$ induce isomorphisms of $\bbP^N$, we have that $tt' (X \star Y)$ and $tX \star t'Y$ are both irreducible. Hence, it is enough to prove that they have the same dimension. 
    
    Since $tt'$ is an isomorphism, $\dim tt' (X \star Y) = \dim X \star Y$. Now, for generic points $p \in X$ and $q \in Y$, $p$ is smooth for $X$, $q$ is smooth for $Y$ and $p\star q$ is smooth for $X \star Y$. Then, $tp$ is smooth for $tX$, $t'q$ is smooth for $t'Y$ and $(tp) \star (t'q) = tt' (p \star q)$ is smooth for $tX \star t'Y$. By Terracini's Lemma for Hadamard products of algebraic varieties (see \cite[Lemma 2.12]{BCK17}), we get
	\begin{align*}
		T_{(tp) \star (t'q)} t X \star t'Y
        &= \langle (tp) \star T_{t'q} t'Y, (t'q) \star T_{tp} tX \rangle\\
		&= \langle (tp) \star (t' T_q Y), (t'q) \star (t T_p X) \rangle\\
		&= tt' \langle p \star T_q Y, q \star T_p X \rangle = tt' T_{p \star q} X \star Y.
	\end{align*}
	Hence, $\dim  t X \star t'Y = \dim X \star Y = \dim tt'( X\star Y)$. This concludes the proof.
\end{proof}

\begin{lemma}
\label{lemma:containsall1}
	Let $X \subset \bbP^N$ be a proper irreducible variety such that $\mathds{1} \in X$. If the generic $X$-Hadamard-rank is infinite, then there exists $r \in \bbZ_{\geq 1}$ such that $X^{\star r}\subsetneq \bbP^N$ is an Hadamard-idempotent subvariety.
\end{lemma}

\begin{proof}
    If the generic $X$-Hadamard-rank is infinite, then the chain in \eqref{eq:chain_inclusions} stabilizes, i.e., there exists $r \in \bbZ_{\geq 1}$ such that $X^{\star r} = X^{\star (r+1)} = \cdots = X^{\star (2r)} = (X^{\star r})^{\star 2} \neq \bbP^N$. Hence $X^{\star r} \subsetneq \bbP^N$ is Hadamard-idempotent.
\end{proof}

\begin{proof}[Proof of \Cref{prop:contained_in_idempotent}]
We distinguish two cases. 

\textit{Case 1.} If $X$ is not concise, then $X \subset H_i$ for some $i = 0, \ldots, N$, and $H_i^{\star 2} = H_i$.

\textit{Case 2.} Suppose that $X$ is concise. Then $X$ contains a point $\bfx = (x_0 : \cdots : x_N) \in X$ such that $\prod_{i=0}^N x_i \neq 0$. If $\bfx = \mathds{1}$, then we conclude by \Cref{lemma:containsall1}. If not, we take $t = \text{diag}(x_0^{-1}, \ldots, x_N^{-1})\in \mathbb{T}$ and then $\mathds{1} \in tX$. Note that the assumption $\Hrk_X^\circ = \infty$ implies that $\eta_r(X) \neq \bbP^N$ for all $r$: in particular, $\dim \eta_r(X) = \dim X^{\star r} < N$ for all $r$. Note that, by \Cref{lemma:diagonal_iso}, $\dim (tX)^{\star r} = \dim t^r X^{\star r} = \dim X^{\star r}$. In particular, we also have that $\Hrk_{tX}^\circ = \infty$. Therefore, $tX$ satisfies the assumptions of \Cref{lemma:containsall1} and this concludes the proof.
\end{proof}

\Cref{prop:contained_in_idempotent} underlines the importance of Hadamard-idempotent varieties. As already mentioned, embedded toric varieties and coordinate hyperplanes are Hadamard-idempotent. In \cite{bocci2022hadamard}, a classification of Hadamard-idempompotent hypersurfaces is given. Here, we employ tools from tropical geometry to extend such result to arbitrary codimension. 

\subsection{Hadamard-idempotent varieties and tropical geometry
}\label{ssec:Hadamard_idempotent}

	Tropicalization associates to every algebraic variety a \textit{polyhedral complex}, i.e., a family of polyhedra which is closed under taking faces and taking intersections. This combinatorial counterpart preserves many invariants. We recall some preliminary definitions and results in polyhedral and tropical geometry. For a general exposition, see \cite{sturmfelsmaclagan2015}. 

    \subsubsection{Structure Theorem of Tropical Varieties}

    In order to understand tropical varieties, one needs to understand first their building blocks, namely \textit{tropical hypersurfaces}. 
    These are given by the \textit{normal fan} to a \textit{Newton polytope}\footnote{Tropical varieties can be defined in a more general setting over fields with \textit{valuations}. Here, we use the trivial valuation.}. Let us recall these definitions. 
    For a more detailed description, see \cite[Chapter 2.3]{sturmfelsmaclagan2015} or \cite[Chapter 1-2]{ziegler2012lectures} for polyhedral geometry background.

    \begin{definition}
    \label{def:Newt-polytope}
        For a polynomial $f = \sum_{\alpha \in \bbZ_{\geq 0}^n} c_\alpha {\bf x}^{\alpha} \in \bbC[x_1, \ldots, x_n]$, the \textbf{Newton polytope} of $f$ is
        \[
            \Newt(f) = \text{conv}(\alpha \in \bbZ_{\geq 0}^n : c_{\alpha} \neq 0 ) \subset \bbR^n.
        \]
    \end{definition}
 
    \begin{example}
    \label{example:newtonpolytopes}
        Given bivariate polynomials $f = x + 2y - 5, \ g = x^3y^2 + 2x^4 + 3xy^3 + 2y - y^2 + 1 \in \bbC[x,y]$, the Newton polytope of $f$ is a triangle while the Newton polytope of $g$ is a pentagon. The Newton polytope of the 3-variate polynomial $h = x + 2y + z - 3 \in \bbC[x,y,z]$ is a simplex in $\bbR^3$.
    \end{example}

    In convex geometry, the dual object to a polytope is its \textit{normal fan}. This collection of cones remembers the information about the orthogonal space to the faces of the polytope. Let us recall some definitions. 

    \begin{definition}
		A \textbf{rational polyhedral cone} in $\mathbb{R}^n$ is the set of positive linear combinations of a finite set of vectors in $\mathbb{Z}^n$, i.e., a set of the form:
		\[
			\text{cone}(\bfv_1, \ldots, \bfv_k) = \{ \lambda_1 \bfv_1 + \cdots + \lambda_k \bfv_k \in \mathbb{R}^n : \lambda_1, \ldots, \lambda_k \geq 0\}, 
		\]
		for some $\bfv_1, \ldots, \bfv_k \in \mathbb{Z}^n$. A \textbf{face} of a rational polyhedral cone $C \subset \mathbb{R}^n$ is a rational polyhedral cone $C' \subset C$ such that if $\bfv,\bfw \in C$ satisfy $\bfv+\bfw\in C'$, then $\bfv,\bfw \in C'$. The dimension of a rational polyhedral cone $C$ is the dimension of its linear span $\langle C\rangle$. We will write:
		\[
			\dim C := \dim \langle C \rangle.
		\]
	\end{definition}

	\begin{definition}
		A \textbf{rational polyhedral fan} in $\mathbb{R}^n$ is a finite collection $\Sigma$ of rational polyhedral cones in $\mathbb{R}^n$ such that:
	\begin{itemize}
		\item For every rational polyhedral cone $C \in \Sigma$ and every face $C'$ of $C$, we have $C' \in \Sigma$;
		\item For every pair of rational polyhedral cones $C_1, C_2 \in \Sigma$ that intersect, $C_1 \cap C_2$ is a face of both $C_1$ and $C_2$.
	\end{itemize}
	The {\bf support} of $\Sigma$ is the set-theoretic union of all its cones $|\Sigma| = \bigcup_{C \in \Sigma} C \subset \mathbb{R}^n$. The \textbf{dimension} of $\Sigma$ is the maximum dimension of a cone $C \in \Sigma$. A rational polyhedral fan $\Sigma$ is \textbf{pure} of dimension $d$ if every maximal-by-inclusion face $C$ of $\Sigma$ has dimension $\dim C = d$. The {\bf $k$-th skeleton} of $\Sigma$ is the polyhedral fan $\Sigma^{\leq k}$ made by the cones $C \in \Sigma$ of dimension $\dim C \leq k$.
	\end{definition}

    An example of a rational polyhedral fan is given by the following construction. 
    \begin{definition}\label{def:normal_fan}
        Let $P = \text{conv}(\bfv_1, \ldots, \bfv_k)$ be a polytope, where $\bfv_1, \ldots, \bfv_k \in \bbZ^n$ and we assume for simplicity that $\bfv_1, \ldots, \bfv_k$ are the vertices of $P$. The (inner) \textbf{normal fan} to $P$ is the rational polyhedral fan $\calN(P)$ with rational polyhedral cones
        \[
            \calN_F(P) = \{ \bfu \in \bbR^n : \bfu \cdot (\bfv - \bfw) \leq 
            0 \text{ for all } \bfv \in \text{conv}(F), \ \bfw \in P\}
        \]
        for each face $\text{conv}(F) \subset P, \ \emptyset \neq F \subset \{\bfv_1, \ldots, \bfv_k\}$.
    \end{definition}

    \begin{example}
    \label{example:tropical_line}
        The normal fan of the Newton polytope of $f=x+2y-5$ in \Cref{example:newtonpolytopes} is a 2-dimensional pure polyhedral fan in $\mathbb{R}^2$ with three rays (i.e., 1-dimensional cones) $\sigma_1, \sigma_2, \sigma_3$ generated by $\bfe_1 = (1,0)$, $\bfe_2 = (0,1)$ and $\bfe_3 = (-1,-1)$ respectively, together with three 2-dimensional cones $C_1 = \text{cone}(\bfe_1,\bfe_2), \ C_2 = \text{cone}(\bfe_1,\bfe_3), \ C_3 = \text{cone}(\bfe_2,\bfe_3)$ and the zero-dimensional cone $\tau = \{(0,0)\}$. As we shall see, the support of its 1-skeleton (i.e., the union of the rays $\sigma_1, \sigma_2, \sigma_3$) is the tropicalization of a generic line $L \subset \bbC^2$ and it is called a {\em tropical line}. 
    \end{example}

    As in the previous example, normal fans of Newton polytopes provide tropical hypersurfaces, as we see in the following definition. Note that, as usually done in the literature, we address tropical geometry in the affine setting. This will allow us anyway to characterize projective varieties with infinite generic Hadamard rank. It will be done by looking at the tropicalization of their affine cones.

    \begin{notation}
        Let $f \in \bbC[x_1, \ldots, x_n]$ be a polynomial. We denote by $Z(f) =\{f=0\}\subset \bbC^n$ the corresponding affine hypersurface. 
    \end{notation}

    \begin{definition}
    \label{def:tropicalhypersurface} 
    Let $f \in \bbC[x_1, \ldots, x_n]$ be a polynomial. 
    The \textbf{tropical hypersurface} $\Trop Z(f) \subset \bbR^n$ is the support of the $(n-1)$-skeleton of the normal fan of the Newton polytope $\Newt(f)$. 
    \end{definition}

    For concise varieties, the previous definition agrees with the usual and more general definition of tropical hypersurfaces using initial ideals or tropical polynomials with respect to the trivial valuation. See \cite[Proposition~3.1.6]{sturmfelsmaclagan2015}.
    
    The general definition of tropicalization relies on tropical hypersurfaces.
    
    \begin{definition}
        Given a concise affine variety $X \subset \bbC^n$, we define its tropicalization as
        \[
            \Trop X = \bigcap_{f \in I(X)} \Trop Z(f) \subset \bbR^n
        \]
        where the intersection runs over all polynomials in the ideal of $X$.
    \end{definition}

    We are ready to state one of the main theorems in tropical geometry, whose main consequence is that tropicalization preserves dimensions.

    \begin{theorem}[Structure Theorem of tropical varieties \cite{sturmfelsmaclagan2015}]
        \label{thrm:structuretheorem}
		Let $X \subset \bbC^N$ be a concise irreducible variety. Then, $\operatorname{Trop} X \subset \mathbb{R}^N$ is the support of a pure rational balanced polyhedral fan of dimension $\dim X$.
	\end{theorem}

    The combinatorial structure of tropical varieties will be our main tool for investigating Hadamard-idempotent varieties. We recall the necessary definitions.

    \begin{notation}
            Let $\Sigma$ be a rational polyhedral fan in $\bbR^n$. We denote by $\Sigma^k$ the set of $k$-dimensional faces of $\Sigma$.
        \end{notation}

	\begin{definition}
		Let $\Sigma$ be a rational polyhedral fan in $\mathbb{R}^n$ of dimension $d$ and $\tau \in \Sigma^{d-1}$. 
        Denote by $\pi_\tau\colon \mathbb{R}^n \twoheadrightarrow \mathbb{R}^n/\langle\tau\rangle$ the projection map.
        The \textbf{star} of $\tau$ in $\Sigma$ is the 1-dimensional rational polyhedral fan $\text{Star}_\Sigma(\tau)$ in $\mathbb{R}^n/\langle\tau\rangle$ whose cones are given by $\pi_\tau(\sigma)$ for every $\sigma \in \Sigma^d$ such that $\sigma \supset \tau$.
	\end{definition}

	\begin{example}
		Consider the tropical line from \Cref{example:tropical_line} as a 1-dimensional fan. We denote it by $\Sigma$. The unique face of codimension 1 is the origin $\tau = \{(0,0)\}$, hence $\pi_\tau\colon \mathbb{R}^2 \twoheadrightarrow \mathbb{R}^2/\langle\tau\rangle \simeq \mathbb{R}^2$ is the identity map and every maximal-dimension cone contains $\tau$, hence $\text{Star}_\Sigma(\tau) = \Sigma$.
	\end{example}

	\begin{definition}
		A \textbf{weighted rational polyhedral fan} of dimension $d$ is a pair $(\Sigma, w)$ where $\Sigma$ is a rational polyhedral fan of dimension $d$ and $w\colon \Sigma^d \to \mathbb{Z}_{>0}$ is a map, called the weight on $\Sigma$.
	\end{definition}

	\begin{example}
		Let $\Sigma$ be a weighted rational polyhedral fan of dimension $d$ in $\mathbb{R}^n$ with weight $w\colon \Sigma^d \to \mathbb{Z}_{>0}$. For every codimension-$1$ face $\tau \in \Sigma^{d-1}$, $\text{Star}_\Sigma(\tau)$ is a 1-dimensional fan in $\mathbb{R}^n/\langle\tau\rangle$. 
        It inherits a weight $w_\tau\colon \text{Star}_\Sigma(\tau)^1 \to \mathbb{Z}_{>0}$ from $\Sigma$ in a canonical way with $w_\tau(\pi_\tau(\sigma)) = w(\sigma)$ for every $\sigma \in \Sigma^d$ such that $\sigma \supset \tau$.
	\end{example}

	\begin{definition}
		Let $\Sigma$ be a weighted 1-dimensional polyhedral fan in $\mathbb{R}^n$ with weight $w\colon \Sigma^1 \to \mathbb{Z}_{>0}$. For each ray $\sigma \in \Sigma^1$, let $\bfe(\sigma) \in \mathbb{Z}^n$ be the first lattice point on $\sigma$. We say that $\Sigma$ is {\bf balanced} if:
		\[
			\sum_{\sigma \in \Sigma} w(\sigma) \bfe(\sigma) = 0 . 
		\]
		For a general weighted rational polyhedral fan $\Sigma$ of dimension $d$ in $\mathbb{R}^n$ with weight $w\colon \Sigma^d \to \mathbb{Z}_{>0}$, we say that $\Sigma$ is balanced if, for every codimension 1 face $\tau \in \Sigma^{d-1}$, the 1-dimensional fan $\text{Star}_\Sigma(\tau)$ is balanced. Alternatively, for every codimension 1 face $\tau \in \Sigma^{d-1}$, let $\bfe_\tau(\sigma)$ be the first lattice point of the 1-dimensional cone $\pi_\tau(\sigma)$ for every $\sigma \in \Sigma^d$ such that $\sigma \supset \tau$. Then, $\text{Star}_\Sigma(\tau)$ is balanced if:
		\[
			\sum_{\substack{\sigma \supset \tau \\ \sigma \in \Sigma^d}} w(\sigma) \bfe_\tau(\sigma) = 0
		\]
		and $\Sigma$ is balanced if every $\text{Star}_\Sigma(\tau)$ is balanced for every $\tau \in \Sigma^{d-1}$.
	\end{definition}

	\begin{example}
		The tropical line from \Cref{example:tropical_line} is balanced with the weight $w(\sigma_i) = 1$ for all $i = 1, 2, 3$. This is a consequence of the \nameref{thrm:structuretheorem}.
	\end{example}

    The \nameref{thrm:structuretheorem} and, in particular, the balancing condition will be important in the study of Hadamard-idempotent varieties. The other main tool is the following theorem about the tropicalization of monomial maps \cite[Corollary 3.2.13]{sturmfelsmaclagan2015}.

    \begin{theorem}
    \label{thrm:monomialmaps}
        Let $\alpha\colon \bbC^n \dashrightarrow \bbC^m$ be a monomial map, i.e., $\alpha({\bf x}) = ({\bf x}^{\alpha_1}, \ldots, {\bf x}^{\alpha_m})$ where $\alpha_1, \ldots, \alpha_m \in \bbZ^n$. Let $X \subset \bbC^n$ be a concise variety and $A \in \bbZ^{m \times n}$ be the matrix whose columns are $\alpha_1, \ldots, \alpha_m$. Then
        \[
            \Trop \overline{\alpha(X)} = A^T (\Trop X).
        \]
    \end{theorem}

    An important consequence regards tropicalizations of Hadamard products of affine varieties $X, Y \subset \bbC^N$.

    \begin{corollary}
    \label{thm:trophadamard}
        Let $X, Y \subset \bbC^N$ be concise affine varieties. Then
        \[
            \Trop X \star Y = \Trop X + \Trop Y , 
        \]
        where the sum on the right denotes the Minkowski sum.
    \end{corollary}
    
    \subsubsection{Hadamard-idempotent varieties and linear tropical varieties}

    The starting point of our investigation is the following remark.

    \begin{remark}
        Toric varieties are Hadamard-idempotent. At the same time, consider a monomial embedding $\widehat{\varphi}_A(\bfx) =(\bfx^{\alpha_0},\ldots,\bfx^{\alpha_N})$ as in \Cref{def:toricvarieties}, whose Zariski closure gives the affine cone of a toric variety $X \subset \bbP^N$. The tropicalization of this embedding gives a parametrization $\bfx \mapsto (\alpha_0 \cdot \bfx, \ldots,\alpha_N \cdot \bfx) = A^T \bfx$ of the tropicalization of the affine cone of $X$, by \Cref{thrm:monomialmaps}. Hence, $\Trop X$ is a linear space.
    \end{remark}

	 Building on the previous remark, it is natural to ask if it is always true that the tropicalization of a Hadamard-idempotent variety is \textit{always} a linear space. In the following, we give a positive answer to this question. As in the previous section, we develop our tools in the affine case. We explain how to pass to projective varieties before stating the main result of this section where we classify concise irreducible projective varieties whose Hadamard powers do not fill the ambient space (\Cref{cor:finite_HadamardRank_condition}).
     
	\begin{lemma}
	\label{lemma:notbalanced}
		Let $\Sigma$ be a polyhedral fan in $\mathbb{R}^n$ whose support is a cone but not a linear space. Then, $\Sigma$ is not balanced with respect to any weight $w\colon \Sigma^d \to \mathbb{Z}_{>0}$.
	\end{lemma}

	\begin{proof}
		Suppose that $\sigma = |\Sigma| = \text{cone}(\bfv_1, \ldots, \bfv_d)$. Up to quotienting by the lineality space of $\sigma$, i.e., the largest 
        linear space contained in $\sigma$, we can suppose that $\sigma \neq \{0\}$ and that $\bfv_1, \ldots, \bfv_d$ are linearly independent. If $d = 1$, then $\Sigma$ is clearly not balanced. If $d \geq 2$, let $\tau = \text{cone}(\bfv_1, \ldots, \bfv_{d-1})$. Then, $\tau$ is a codimension-1 cone of $\Sigma$ and $\text{Star}_\Sigma(\tau)$ is not balanced. In fact, since $\bfv_d \not \in \tau$, we have that
        $\pi_\tau(\bfv_d) \neq \bf0$ and $\text{Star}_\Sigma(\tau)^1 = \{ \text{cone}(\pi_\tau(\bfv_d))\}$. A single ray is not balanced, hence $\text{Star}_\Sigma(\tau)$ is not balanced.
	\end{proof}

	\begin{remark}
		In the proof of \Cref{lemma:notbalanced}, we use the hypothesis that $\sigma = |\Sigma|$ is a cone but not a linear space. In fact, if $\sigma$ is a linear space, then it is equal to its lineality space, and then by quotienting we obtain $\{\bf0\}$ and we cannot continue with the proof. A way of writing a linear space as a cone is by taking a basis $\bfv_1, \ldots, \bfv_d$ and writing it as $\text{cone}(\pm \bfv_1, \ldots, \pm \bfv_d)$. But then, the set $\{ \pm \bfv_1, \ldots, \pm \bfv_d \}$ is not linearly independent. Any other generating set (as a cone) will give rise to this problem.
	\end{remark}

	\begin{theorem}
	\label{thrm:tropXlinear}
		Let $X \subset \bbC^N$ be a concise irreducible affine variety such that $X^{\star 2} = X$. Then, $\Trop X$ is a linear space. 
	\end{theorem}
	\begin{proof}
		By \Cref{thm:trophadamard} and the Hadamard-idempotent assumption, we know that $\Trop X + \Trop X = \Trop X^{\star 2} = \Trop X$. By induction, for any $k\geq 2$:
		\begin{equation}
		\label{eqn:ktrop}
			k \cdot \Trop X = \underbrace{\Trop X + \cdots + \Trop X}_{k \text{ times}} = \Trop X.			
		\end{equation}
		By the \nameref{thrm:structuretheorem}, $\Trop X$ is (the support of) a pure, balanced rational polyhedral fan $\Sigma$ of dimension $\dim X$. Let $\bfv_1, \ldots, \bfv_k \in \mathbb{Z}^N$ be the collection of all rays that appear among a generating set of a cone in $\Sigma$. We shall prove that $\Trop X = \text{cone}(\bfv_1, \ldots, \bfv_k)$. Take $\lambda_1, \ldots, \lambda_k \geq 0$. Since $\Trop X$ is the support of a polyhedral fan, $\lambda_i \bfv_i \in \Trop X$ for any $i \in \{1, \ldots, k\}$. Thus, by \Cref{eqn:ktrop}, we have that:
		\[
			\underbrace{\lambda_1 \bfv_1}_{\in \Trop X} + \cdots + \underbrace{\lambda_k \bfv_k}_{\in \Trop X} \in k\cdot \Trop X = \Trop X,
		\]
		from which we obtain that $\text{cone}(\bfv_1, \ldots, \bfv_k) \subset \Trop X$. Moreover, $\Trop X \subset \text{cone}(\bfv_1, \ldots, \bfv_k)$ always holds true. In fact, since $\Trop X$ is the support of $\Sigma$, for every $\bfx \in \Trop X$, there exists a cone $C \in \Sigma$ such that $\bfx \in C$. By definition of the vectors $\bfv_1, \ldots, \bfv_k \in \mathbb{Z}^n$, there exist $1\leq i_1, \ldots, i_m \leq k$ for which $C = \text{cone}(\bfv_{i_1}, \ldots, \bfv_{i_m})$. This means that $\bfx \in C \subset \text{cone}(\bfv_1, \ldots, \bfv_k)$. In the end, we obtain that $\Trop X$ is a cone. Since $\Sigma$ is balanced, by \Cref{lemma:notbalanced}, $\Trop X$ (i.e., the support of $\Sigma$) has to be a linear space.
	\end{proof}

    \begin{notation}
        In what follows, by a binomial we mean a polynomial of the form $a \bfx^\alpha + b \bfx^\beta$ for some $a, b \in \bbC$, $a,b \neq 0$ and $\alpha, \beta \in \bbZ_{\geq 0}^N$.
    \end{notation}

    \begin{proposition}
	\label{prop:alignedsupport}
		Let $Z(f) \subset \bbC^N$ be a concise irreducible hypersurface. Then, $\Trop Z(f)$ is a linear space if and only if $f$ is a binomial.
	\end{proposition}
	\begin{proof}
		  If $f = a \bfx^{\alpha} + b \bfx^{\beta}$ is a binomial, then $\Newt(f)$ (\Cref{def:Newt-polytope}) is the line segment with 
          vertices $\alpha,\beta \in \bbZ_{\geq 0}^N$ and (the support of) its normal fan (\Cref{def:normal_fan}) is the tropical hypersurface $\Trop Z(f) = \{ \mathbf{x} \in \bbR^N : (\alpha-\beta) \cdot \mathbf{x} = 0 \}$, which is an hyperplane.
          
          Assume that $\Trop Z(f)$ is a linear space. Let $\text{Newt}(f) \cap \bbZ_{\geq 0}^N = \{\alpha_0, \ldots, \alpha_m\}$ be the integer points of $\text{Newt}(f)$; that is, all the possible exponents of monomials that can appear in $f$, i.e.,
		\[
			f(\mathbf{x}) = c_0 \mathbf{x}^{\alpha_0}+ c_1 \mathbf{x}^{\alpha_1} + \cdots + c_m \mathbf{x}^{\alpha_m} \in  \bbC[\mathbf{x}]
		\]
		for some $c_0, \ldots, c_m \in \bbC$. By the definition of tropical hypersurface (\Cref{def:tropicalhypersurface}), we know that $\Trop Z(f)$ is (the support of) the $(n-1)$-skeleton of the normal fan of the Newton polytope $\text{Newt}(f)$. Hence, from the assumptions, $\text{Newt}(f)$ is a line segment. In particular, we can assume that $\alpha_0$ and $\alpha_m$ are the vertices of $\text{Newt}(f)$ and that $\alpha_1$ is the first lattice point after $\alpha_0$ on the oriented segment from $\alpha_0$ to $\alpha_m$. In other words, we have that $\alpha_j - \alpha_0 = j\beta$ for all $j$, where $\beta = \alpha_1 - \alpha_0$. Then, we can then rewrite
		\[
			f(\mathbf{x}) = \mathbf{x}^{\alpha_0}(c_0 + c_1 \mathbf{x}^\beta + c_2 (\mathbf{x}^\beta)^2 + \cdots + c_m (\mathbf{x}^\beta)^m),
		\]
        where the second factor can be regarded as a univariate polynomial in $\bfx^\beta$ with complex coefficients and, therefore, we can decompose
		\begin{equation}
        \label{eq:decomposition}
		    f(\mathbf{x}) = c \ \mathbf{x}^{\alpha_0} (\mathbf{x}^\beta - \lambda_1) \cdots (\mathbf{x}^\beta - \lambda_m)
		\end{equation}
        for some $c, \lambda_1,\ldots,\lambda_m\in \bbC$. In order to conclude the proof it is enough to show that $m = 1$. Tne o do that, we employ the irreducibility of $f$. Note that, \Cref{eq:decomposition} is a decomposition as Laurent polynomials.

        We denote $\alpha_i = (\alpha_{i,1},\ldots,\alpha_{i,n})$ and $\beta = (\beta_1,\ldots,\beta_n).$

        Let $\bar{\alpha} = (\min_{i=0,\ldots,m} \alpha_{i,j})_{j=1,\ldots,N}$. Since $\Newt(f)$ is a segment, $\bar{\alpha}_j = \min(\alpha_{0,j},\alpha_{m,j})$. Clearly, $\bfx^{\bar{\alpha}} \ | \ f$. Since $f$ is assumed to be irreducible, $\bar{\alpha} = {\bf 0}$. Therefore, we will distinguish two cases: (i) $\bar{\alpha}=\alpha_0$ and (ii) $\bar{\alpha}\neq \alpha_0$.

        \textit{Case (i).} If $\bar{\alpha} = \alpha_0$, then the conclusion is immediate. Indeed, $\beta = \alpha_1 \in \bbZ_{\geq 0}^N$ and, by irreducibility, $m = 1$.

        \textit{Case (ii).} Assume $\bar{\alpha} \neq \alpha_0$. Since $\alpha_m - \alpha_0 = m \beta$, then for any $i\in \{1,\ldots,N\}$, $\alpha_{m,i} - \alpha_{0,i} = m \beta_i$. Note that, for any $i$, $\beta_i <0$ if and only if $\alpha_{0,i} > 0$ if and only if $\alpha_{m,i} = 0$ (the latter equivalence is because $\bar{\alpha}=\bf0$). In particular, $\alpha_0 = m \beta^-$, where $\beta^- = (\max (-\beta_j,0))_{j=1,\ldots,N}$ and, from \Cref{eq:decomposition}, we deduce the polynomial decomposition
        \[
            f(\bfx) = c \ (\bfx^{\beta + \beta^-} - \lambda_1\bfx^{\beta^-})\cdots (\bfx^{\beta + \beta^-} - \lambda_m\bfx^{\beta^-})
        \]
        where $\beta^-, \beta + \beta^- \in \bbZ_{\geq 0}^N$. By irreducibility, $m = 1$ and we conclude.
	\end{proof}
    
	\begin{corollary}
	\label{cor:perpendicular}
		Let $Z(f) \subset \bbC^N$ be a concise irreducible hypersurface such that $(\Trop Z(f))^\perp \neq \{ \bf0 \}$. Then, $f$ is a binomial.
	\end{corollary}
	\begin{proof}
			Let $F = \Trop Z(f)$. Since $F^\perp \neq \{\bf0\}$ and $\dim \langle F \rangle \geq N-1$, we have that $\dim F^\perp = 1$. Hence, $\dim \langle F \rangle = N-1$. Now, suppose towards obtaining a
            contradiction that $f$ is not a binomial. By \Cref{prop:alignedsupport}, $F$ is not a linear space. In particular, the Newton polytope of $f$ cannot be a line segment, hence the support of $f$ is not made by aligned vectors. Write $f = c_0 \bfx^{\alpha_0} + \cdots + c_d\bfx^{\alpha_d}$. Then, we have that there exist $i, j, k$ such that $\alpha_i - \alpha_j$ and $\alpha_i - \alpha_k$ are linearly independent. Hence, the following two faces of $F$ are contained in two distinct hyperplanes of $\bbR^N$: 
		\[
			F_{ij} = \{ \mathbf{x} \in \bbR^N : \alpha_i \cdot \mathbf{x} = \alpha_j \cdot \mathbf{x} \leq \alpha_\ell \cdot \mathbf{x} \ \forall \ell\}, \quad F_{ik} = \{ \mathbf{x} \in \bbR^N : \alpha_i \cdot \mathbf{x} = \alpha_k \cdot \mathbf{x} \leq \alpha_\ell \cdot \mathbf{x} \ \forall \ell\}.
		\]
		This implies that $\langle F \rangle = \bbR^N$, a contradiction.
	\end{proof}

	\begin{theorem}
		\label{thrm:containedinbinomial}	
		Let $X \subset \bbC^N$ be a concise irreducible variety. Then, $\Trop X$ is contained in a proper linear subspace of $\bbR^N$ if and only if $X$ is contained in a binomial hypersurface of $\bbC^N$.
	\end{theorem}
	\begin{proof}
		If $X \subset Z(f)$ and $f$ is a binomial, then $\Trop X \subset \Trop Z(f)$ and $\Trop Z(f)$ is a proper linear subspace by \Cref{prop:alignedsupport}. For the opposite direction, let $\Sigma = \Trop X$. By \cite[Theorem 2.6.6]{sturmfelsmaclagan2015}, every ideal has a finite \textit{tropical basis}. This means that there exists a finite collection $f_1, \ldots, f_s$ of generators of the ideal of $X$ such that
		\[
			\Sigma = \Sigma_1 \cap \cdots \cap \Sigma_s, \quad \Sigma_i = \Trop Z(f_i).
		\]
		This implies that
		\[
			\Sigma^\perp = \Sigma_1^\perp + \cdots + \Sigma_s^\perp.
		\]
		Since $X$ is irreducible, each $f_i$ is irreducible. Moreover, since $\Sigma$ is contained in a proper linear subspace, $\Sigma^\perp \neq \{\bf0\}$, from which it follows that $\Sigma^\perp_i \neq \{\bf0\}$ for some $i \in \{1,\ldots, s\}$. Since $\Sigma_i$ is a tropical hypersurface, by \Cref{cor:perpendicular} we know that $f_i$ is a binomial. Hence, $X$ is contained in a binomial hypersurface $Z(f_i)$.
	\end{proof}

    \begin{remark}\label{rmk:projective_tropicalization}
            The previous results work also in the projective setting. Given a concise projective variety $X \subset \bbP^n$, denote by $\widehat{X} \subset \bbC^{n+1}$ its affine cone. Then, one defines the tropicalization of $X$ by quotienting $\Trop \widehat{X}$ by the span of $\mathds{1}$ in order to work in the {\em tropical projective space} $\mathbb{TP}^n$: 
            \[
            \Trop X = \Trop \widehat{X} / \langle \mathds{1} \rangle \subset \mathbb{TP}^n = \bbR^{n+1}/\langle \mathds{1} \rangle.
            \]
            In particular, in the previous results, $\Trop X$ is a linear space inside $\mathbb{TP}^n \simeq \bbR^n$ if and only if $\Trop \widehat{X}$ is a linear space in $\bbR^{n+1}$.
        \end{remark}

    We conclude this section with the classification of irreducible varieties with infinite generic $X$-Hadamard-rank in the projective setting, i.e., such that Hadamard powers of $X$ do not fill the ambient space.
	\begin{corollary}\label{cor:finite_HadamardRank_condition}
		Let $X \subset \bbP^N$ be a concise irreducible variety. Then, the following are equivalent:
        \begin{enumerate}
            \item the generic $X$-Hadamard-rank is infinite, i.e., $X^{\star m} \subsetneq \bbP^N$ for all $m \geq 1$;
            \item $X$ is contained in a proper binomial hypersurface;
            \item the ideal of $X$ contains binomials.
        \end{enumerate}
        \end{corollary}
	\begin{proof}
		The fact that (2) is equivalent to (3) is trivial by the fact that taking ideals reverses inclusions. 
        
        We prove that (2) implies (1). Assume that $X \subset Y \subsetneq \bbP^N$ where $Y$ is a binomial hypersurface. Interpreting \Cref{prop:alignedsupport} in light of \Cref{rmk:projective_tropicalization}, $\Trop Y$ is a linear space. Hence, by \Cref{thm:trophadamard}, $\Trop Y^{\star m} = \Trop Y + \cdots + \Trop Y = \Trop Y$. Hence, $\dim X^{\star m} \leq \dim Y^{\star m} = \dim Y < N$ and $\Hrk_X^\circ = \infty$.
        
        We now prove that (1) implies (2). Since $X$ is concise, then by \Cref{prop:contained_in_idempotent}, we know that $X \subset t Y$, where $t \in \bbT$ is a diagonal invertible matrix, $Y = Y^{\star 2}$ and $Y$ is concise and irreducible. By \Cref{thrm:tropXlinear}, we know that $\Trop Y$ is a linear space. Hence, by \Cref{thrm:containedinbinomial}, $Y$ is contained in an irreducible binomial hypersurface $Z$. In particular, $t Z$ is still an irreducible binomial hypersurface. Since $X \subset tY \subset tZ$, this concludes the proof.
	\end{proof}

    \begin{remark}
        We recall that the task of checking computationally whether an ideal in a polynomial ring contains a binomial was studied in \cite{jensen2017finding}.
    \end{remark}

\section{Dimensions of Hadamard products of secant varieties}\label{sec:dimension_Hadamard_powers}

Inspired by \Cref{question:B} and the connection with Restricted Boltzmann Machines already mentioned in \Cref{rmk:RBM}, we are interested in the dimensions of Hadamard products of secant varieties. 
This allows to compute generic $r$th Hadamard-X-ranks: indeed, recall that the generic $r$th Hadamard-$X$-rank with respect is equal to the smallest $m$ for which $\sigma_r(X)^{\star m}$ fills the ambient space. Since many varieties appearing in tensor decomposition and algebraic statistics are secant varieties of toric varieties, we will focus on them.

\begin{notation}
    In order to ease notation, for any $\bfr = (r_1,\ldots,r_m)$, we write $\sigma_\bfr(X) := \sigma_{r_1}(X)\star\cdots\star\sigma_{r_m}(X)$.
\end{notation}

\subsection{Expected dimensions}
Secant varieties of non-degenerate algebraic varieties have a notion of \textit{expected dimensions} which comes from a naïve parameter count. If $X \subset \bbP^N$ is non-degenerate, then 
\begin{equation}\label{eq:expdim_secants}
	\exp \!.\dim\sigma_r(X) = \min\{N, \ r\dim X + r-1\} \quad \text{ and }\quad \dim \sigma_r(X) \leq \exp \!.\dim\sigma_r(X).
\end{equation}
When the latter inequality is strict we say that $X$ is \textit{$r$-defective}. 

Examples of defective varieties are known among Segre-Veronese varieties since the XIX century. The challenge of providing a complete classification of such defective varieties took the attention of a very broad literature since the late XIX century. Complete results are known for Veronese varieties and some families of Segre-Veronese varieties. We refer to \cite{oneto2025ranks} for a recent updated list of known cases.

In the case of Hadamard products of algebraic varieties a notion of expected dimension was introduced in \cite{BCK17}. In particular, given $X,Y \subset \bbP^N$,
\begin{equation}
\label{eqn:expected_dimension_Hadamard}
    \exp \!.\dim X \star Y = \min\{\dim X + \dim Y - \dim T, \ N\} , 
\end{equation}
where $T \subset (\bbC^\times)^{N+1}/\bbC^\times$ is the highest-dimensional torus acting on both $X$ and $Y$. By \cite[Proposition~5.4]{BCK17}, $\dim X \star Y  \leq \exp \!.\dim X\star Y$.
In particular, if $X \subset \bbP^N$ is a toric variety, that is, we have a torus $\mathbb{T}_X \simeq (\bbC^\times)^{\dim X}$ acting on $X$, then, from \Cref{eqn:expected_dimension_Hadamard}, we get
\begin{equation}\label{eq:expdim_hadamard}
	\dim \sigma_\bfr(X)\leq \exp.\dim\sigma_\bfr(X) = \min\left\{\sum_{i=1,\ldots,m}\dim\sigma_{r_i}(X) - (m-1) \dim X, N\right\}.
\end{equation} 
If the inequality is strict, we say that $X$ is \textbf{\textit{$\bfr$}-Hadamard-defective}.

If we bound \Cref{eq:expdim_hadamard} by using the expected dimension of $\sigma_{r_i}(X)$ from \Cref{eq:expdim_secants}, then 
\begin{align}
	\dim \sigma_\bfr(X) & \leq \min\left\{\sum_{i=1,\ldots,m}(r_i\dim X + r_i-1) - (m-1)\dim X, N\right\} \label{eq:expected_dim_equality} \\ 
	& = \min\left\{\left(\sum_{i=1,\ldots,m}(r_i-1)+1\right)\dim X + \sum_{i=1,\ldots,m}(r_i-1),N\right\} \nonumber \\
    & = \exp \!.\dim\sigma_{R}(X), \qquad \text{ with } R := R(r_1,\ldots,r_m) = \sum_{i=1,\ldots,m}(r_i-1)+1.\nonumber 
\end{align} 
In the case $r_1 = \cdots = r_m = r$, expecting \Cref{eq:expected_dim_equality} to be an equality, namely, that $\dim \sigma_r(X)^{\star m} = \min\left\{\left(m(r-1)+1\right)\dim X + m(r-1),N\right\}$, we say that the \textit{expected generic $r$th Hadamard-$X$-rank} with respect to a toric variety $X$ is
\[
    \exp \!. \ \Hrk_{X,r}^\circ = \left\lceil \frac{N-\dim X}{(r-1)(\dim X+1)} \right\rceil.
\]
In \Cref{lemma:upperbound_dimensions}, we will show that the dimension of the $R$th secant variety as in \Cref{eq:expected_dim_equality} of a toric variety $X$ is always a lower bound for the dimension of $\sigma_{\bfr}(X)$. Explicitly, continuing from \Cref{eq:expected_dim_equality}, we will prove that, given a concise toric variety $X$ 
\begin{equation}
\label{eq:chain_of_dimensions}
    \dim\sigma_{R}(X) \leq \dim \sigma_{\bfr}(X) \leq \exp \!.\dim\sigma_{\bfr}(X) = \exp \!.\dim\sigma_{R}(X).
\end{equation}
In the cases where we know that $X$ is not $R$-defective, this chain of inequalities is a chain of equalities and allows us to compute dimensions of Hadamard products of secant varieties of $X$. 

This is the strategy used in \cite[Corollary 26]{MM17:DimensionKronecker} for restricted binary Boltzmann machines, that is, in the case of $X = S_{\mathds{1}}$ being the Segre variety of binary tensors and $r_1 = \cdots = r_m = 2$, for any $m$, employing the classification of defective binary Segre varieties $S_{\mathds{1}}$ from \cite{catalisano2011secant}. With \Cref{lemma:upperbound_dimensions}, we extend this strategy to any concise toric variety for an arbitrary choice of $\bfr$.

\subsection{A useful parametrization} We introduce a useful parametrization of (an affine chart) of the Hadamard product of secant varieties of a toric variety $X \subset \bbP^N$ that we will use instead of the classical one.

Let $A = (\alpha_0 | \cdots | \alpha_N) \in \mathbb{Z}^{(n+1)\times(N+1)}$ be a matrix giving a monomial parametrization $\varphi_A\colon (\bbC^\times)^{n+1} \to \bbP^N$ of the toric variety $X \subset \bbP^N$ as in \Cref{def:toricvarieties}. 

\subsubsection{Parametrizations of $r$th secant varieties}
The secant variety $\sigma_r(X) \subset \bbP^N$ can be regarded as a linear projection of the toric variety defined by the matrix $A \otimes I_r \in \mathbb{Z}^{r(n+1)\times r(N+1)}$, where $I_r \in \mathbb{Z}^{r \times r}$ is the identity matrix. 
Indeed, if $\varphi_A \colon (\bbC^\times)^{n+1} \longrightarrow \bbP^N, \ {\mathbf{x}} \mapsto (\cdots: {\mathbf{x}}^{\alpha_i}:\cdots)$ defines $X$, then $\sigma_r(X)$ is parametrized by the composition of the maps
\begin{equation} 
    \begin{array}{c c c c}
    \varphi_{A \otimes I_r} : & ((\bbC^\times)^{n+1})^{\times r} & \longrightarrow &
     \bbP^{(N+1)r -1} , \\ 
    & ( {\mathbf{x}}_1,\ldots, {\mathbf{x}}_r) & \longmapsto &(\cdots:\varphi_A(\bfx_i):\cdots)_{i=1,\ldots,r}; 
    \\
    \\
    \pi_r : & \bbP^{(N+1)r-1}
    & \dashrightarrow & \bbP^N \\
    & ((z_{i,0}:\cdots:z_{i,N}))_{i=1,\ldots,r} & \longmapsto & (\cdots:\sum_{i=1}^r z_{i,j}:\cdots)_{j = 0,\ldots,N}.
    \end{array}
\end{equation} 
On the affine side, this corresponds to the usual parametrization of the affine cone of the $r$th secant variety as
\begin{equation}
    \begin{array}{c c c c}
     \widehat{\Phi}_{A \otimes I_r} : &((\bbC^\times)^{n+1})^{\times r} & \longrightarrow & \bbC^{N+1}, \\ 
    & ( {\mathbf{x}}_1,\ldots, {\mathbf{x}}_r) & \longmapsto & \widehat{\varphi}_A(\bfx_1) + \widehat{\varphi}_A(\bfx_2) + \cdots + \widehat{\varphi}_A(\bfx_r).
    \end{array}
    \label{eq:classical-secant}
\end{equation}
If we pre-compose the above map with the automorphism of the torus $((\bbC^\times)^{n+1})^{\times r}$ given by 
\[
(\bfy_1,\ldots,\bfy_r) \longmapsto (\bfy_1,\bfy_1\star\bfy_2,\ldots,\bfy_1 \star \bfy_r),
\] 
whose inverse is $(\bfx_1,\ldots,\bfx_r) \longmapsto (\bfx_1,\bfx_1^{\star (-1)}\star\bfx_2, \ldots, \bfx_1^{\star (-1)}\star \bfx_r)$, we get a parametrization of the affine cone of the $r$th secant variety as
\begin{equation}
    \begin{array}{c c c c}
     \widehat{\Phi}_{A \otimes \bar{I}_r} : &((\bbC^\times)^{n+1})^{\times r} & \longrightarrow & \bbC^{N+1}, \\ 
    & ( {\mathbf{y}}_1,\ldots, {\mathbf{y}}_r) & \longmapsto & \widehat{\varphi}_A(\bfy_1) + \widehat{\varphi}_A(\bfy_1\star\bfy_2) + \cdots + \widehat{\varphi}_A(\bfy_1\star\bfy_r),
    \end{array}
\label{eq:reparametrization_secant}
\end{equation}
namely, as the composition of the monomial map $\widehat\varphi_{A \otimes \bar{I}_r}$, where 
$
    \bar{I}_r = \left(\begin{smallmatrix}
        1 & \mathds{1} \\
        \bf0 & I_{r-1}
    \end{smallmatrix}\right)
$, 
with the affine linear projection corresponding to $\pi_r$, which we denote by $\widehat\pi_r$. As mentioned in \Cref{rmk:reparametrize_toric}, (dense subsets of) toric varieties can be reparametrized in terms of matrices that have the same row-span as the original matrix. 
This is the case when we replace $A \otimes I_r$ with $A \otimes \bar{I}_r$.

\begin{example}\label{example:running_1}
    As a running example we consider the case $\bfr = (2,3)$ that we will use to illustrate the definitions and the proof of \Cref{lemma:upperbound_dimensions} presented below. 
    Namely, we will compare the dimension of $\sigma_4(X)$ and the dimension of $\sigma_{(2,3)}(X) = \sigma_2(X) \star \sigma_3(X)$. In particular, we consider the parametrization of (a dense subset of) the affine cone of the $4$th secant variety of $X$ given by 
    \[
        \widehat{\Phi}_{A\otimes \bar{I}_4}(\bfy_1,\bfy_2,\bfy_3,\bfy_4) = \widehat{\varphi}_A(\bfy_1) + \widehat{\varphi}_A(\bfy_1\star\bfy_2) + \widehat{\varphi}_A(\bfy_1\star\bfy_3) + \widehat{\varphi}_A(\bfy_1\star\bfy_4).
    \]
\end{example}

\subsubsection{Parametrization of Hadamard products of secant varieties} 
The Hadamard product of secant varieties $\sigma_{\bfr}(X) = \sigma_{r_1}(X) \star\cdots\star \sigma_{r_m}(X)\subset \bbP^N$ is a linear projection of the toric variety defined by the matrix $A \otimes B_{\bfr'}\in \bbZ_{\geq 0}^{(\sum_{k}r_k)(n+1) \times (\prod_k r_k)(N+1)}$, where $\bfr' := \bfr - \mathds{1}$, i.e., $r'_k = r_k-1$, and $B_{\bfr'}\in \mathbb{Z}^{(\sum_k r_k) \times (\prod_k r_k)}$ is the matrix defining the standard monomial parametrization of the Segre variety $S_{\bfr'}$ as in \Cref{example:reparametrize_segre}. 
Indeed, $\sigma_{\bfr}(X) \subset \bbP^N$ is parametrized by the composition of the maps  
\begin{equation} 
    \begin{array}{c c c c}
    \varphi_{A \otimes B_{\bfr'}} \colon & \prod_{k=1}^m((\bbC^\times)^{n+1})^{\times r_k} & \longrightarrow & 
     \bbP^{(N+1)\times(\prod_k r_k) -1}, \\ 
    & ({\mathbf{x}}_{k,i})_{\substack{k=1,\ldots,m \\ i = 0,\ldots,r'_k}} & \longmapsto &(\cdots :\varphi_A( \mathbf{x}_{1,i_1} \star \cdots \star \mathbf{x}_{m,i_m}) :\cdots)_{\mathbf{i}\in [r'_1]\times\cdots\times [r'_m]};
    \\
    \\
\pi_{\bfr} \colon &  \bbP^{(N+1)\times(\prod_k r_k) -1} &\dashrightarrow& \bbP^{N} \\
& ((z_{\bfi,0}:\ldots:z_{\bfi,N}))_{\bfi\in [r'_1]\times\cdots\times [r'_m]} &\longmapsto & (\cdots:\sum_{\bfi \in [r'_1]\times\cdots\times[r'_k]}z_{\bfi,j}:\cdots)_{j = 0,\ldots,N}, 
    \end{array}
    \label{eq:par-sigma-r}
\end{equation}
where $[r'_k]:=\{0,1,\ldots, r'_k\}$.  
On the affine side, this corresponds to the classical parametrization of the affine cone of $\sigma_\bfr(X)$ given by 
\begin{equation}
    \begin{array}{c c c c}
     \widehat{\Phi}_{A \otimes B_{\bfr'}} : &\prod_{k=1}^m((\bbC^\times)^{n+1})^{\times r_k} & \longrightarrow & \bbC^{N+1}, \\ 
    & ( {\mathbf{x}}_{k,i_k})_{\substack{k=1,\ldots,m \\ i_k = 0,\ldots,r'_k}} & \longmapsto & 
    \left(\sum_{i_1=0}^{r'_1}\widehat{\varphi}_A(\bfx_{1,i_1})\right)\star\cdots\star\left(\sum_{i_m=0}^{r'_m} \widehat{\varphi}_A(\bfx_{m,i_m})\right)\\
    & & & = \sum_{\bfi \in [r'_1]\times\cdots\times [r'_m]}\left(\widehat{\varphi}_A(\bfx_{1,i_1}) \star \cdots \star \widehat{\varphi}_A(\bfx_{m,i_m})\right).
    \end{array}
\end{equation}
Since $\widehat{\varphi}_A(\bfx_1 \star \bfx_2) = \widehat{\varphi}_A(\bfx_1) \star \widehat{\varphi}_A({\bfx_2})$, 
we have that 
\[
    \mathop{\bigstar}_{k=1}^m\left(\sum_{i_k=0}^{r'_k}\widehat{\varphi}_A(\bfx_{k,i_k})\right)
    =  \widehat{\varphi}_A(\bfx_{1,0}\star \bfx_{2,0}\star\cdots\star\bfx_{m,0}) \star \mathop{\bigstar}_{k=1}^m\left(\sum_{i_k=0}^{r'_k}\widehat{\varphi}_A(\bfx_{k,i_k}\star \bfx_{k, 0 }^{\star (-1)})\right). 
\]
Using this and the fact that $\widehat{\varphi}_A(\bfx \star \bfx^{\star(-1)}) = \widehat{\varphi}_A(\mathds{1}) = \mathds{1}$,  
we obtain a parametrization of the affine cone of $\sigma_{\bfr}(X)$ as 
\begin{equation}
    \begin{array}{c c c c}
     \widehat{\Phi}_{A \otimes \bar{B}_{\bfr'}} : & ((\bbC^\times)^{n+1})^{\times \left(\sum_kr'_k+1\right)} & \longrightarrow & \bbC^{N+1}, \\ 
    & 
    (\mathbf{y}_0,
    ( \bfy_{k,j})_{\substack{k=1,\ldots,m \\ j = 1,\ldots,r'_k}} ) & \longmapsto 
    & \widehat{\varphi}_A(\bfy_0) \star \mathop{\bigstar}_{k=1}^m\left(\mathds{1} + \sum_{i_k=1}^{r'_k} \widehat{\varphi}_A(\bfy_{k,i_k})\right) \\
    & & & = \sum_{\bfj \in [r'_1]\times\cdots\times [r'_m]} \widehat{\varphi}_A(\bfy_0 \star \bfy_\bfj) , 
    \end{array}
\label{eq:reparametrize_hadamard} 
\end{equation}
where
\[
    \bfy_\bfj := \bfy_{1,j_1} \star \cdots \star \bfy_{m,j_m}, \text{ with the convention that } \bfy_{i,0} = \mathds{1}.
\]
Let $\bar{B}_{\bfr'}$ be the matrix defined in \Cref{example:reparametrize_segre}. Observe that \eqref{eq:reparametrize_hadamard} is obtained by composing the parametrization $\widehat{\varphi}_{A \otimes \bar{B}_{\bfr'}}$ with the affine linear projection corresponding to $\pi_{\bfr}$, which we denote by $\widehat\pi_{\bfr}$. 
Moreover, as expected, in the special case $m = 1$, \eqref{eq:reparametrization_secant} and \eqref{eq:reparametrize_hadamard} coincide (up to relabeling of the parameter indices).

Summing up, we have a parametrization of the affine cone of $\sigma_r(X)$ obtained as $\widehat\pi_r \circ \widehat\varphi_{A \otimes \bar{I}_r}$ and a parametrization of the affine cone of $\sigma_\bfr(X)$ obtained as $\widehat\pi_{\bf r} \circ \widehat\varphi_{A \otimes \bar{B}_{\bfr'}}$. 
Note that the matrix $\bar{B}_{\bfr'}$ contains the matrix $\bar{I}_{R}$, $R=\sum_{k=1}^mr'_k+1$ as a submatrix. This will be crucial for our proof of \Cref{lemma:upperbound_dimensions} in the next section. 

\begin{example}\label{example:running_2}
    Continuing our running example (\Cref{example:running_1}), we consider the parametrization of (a dense subset of) the affine cone of the Hadamard product $\sigma_{(2,3)}(X) = \sigma_2(X)\star\sigma_3(X)$ given by
    \begin{align*}
        &\widehat{\Phi}_{A\otimes \bar{B}_{1,2}}(\bfy_0,\bfy_{1,1},\bfy_{2,1},\bfy_{2,2}) = \widehat{\varphi}_A(\bfy_0) \star (\mathds{1} + \widehat{\varphi}_A(\bfy_{1,1})\star (\mathds{1} + \widehat{\varphi}_A(\bfy_{2,1}) + \widehat{\varphi}_A(\bfy_{2,2})) \\
         & = \widehat\varphi_A(\bfy_0) + \widehat\varphi_A(\bfy_0\star\bfy_{1,1}) + \widehat\varphi_A(\bfy_0\star\bfy_{2,1}) + \widehat\varphi_A(\bfy_0\star\bfy_{2,2}) + \widehat\varphi_A(\bfy_0\star\bfy_{1,1} \star \bfy_{2,1}) + 
        \widehat\varphi_A(\bfy_0\star\bfy_{1,1} \star \bfy_{2,2}).
    \end{align*}
\end{example}

\subsection{On dimensions of Hadamard products of secant varieties of toric varieties}

In this section, we prove the following result. As already mentioned, this will be our crucial tool to deduce dimensions of Hadamard products of secant varieties of toric varieties whose secants are known to be non-defective. 

\begin{theorem}
\label{lemma:upperbound_dimensions} 
    Let $X$ be a non-degenerate toric variety. Let $\bfr = (r_1,\ldots,r_m) \in \bbZ_{\geq 1}^m$ and let $R = \sum_{k=1}^m (r_k-1)+1$. Then, the dimension of $\sigma_R(X)$ is always a lower bound for the dimension of $\sigma_\bfr(X)$. 
\end{theorem}
\begin{remark}
    As mentioned, this generalizes \cite[Lemma 25]{MM17:DimensionKronecker}, which was only for $r_1 = \cdots = r_m = 2$. Except for small adaptations, the strategy of the proof is essentially the same. 
    Note that the discussion of \cite{MM17:DimensionKronecker} is in terms of \textit{exponential families}. 
    However, it is well-known that such statistical models are strictly related to embedded toric varieties. 
    Indeed, the standard exponential parametrization of the exponential family associated with the matrix $A \in \bbZ^{(n+1)\times(N+1)}$ corresponds to the precomposition of our monomial map $\widehat\varphi_A$ with the map $\bbR^{n+1} \to \bbR_{\geq 0}^{n+1}$ given by $(\eta_1,\ldots,\eta_n) \mapsto (\exp(\eta_1),\ldots,\exp(\eta_n))$, see also \cite[Example~2.4]{michalek2016exponential}. 
    We further observe that the argument of the proof and the statement remain valid when the matrix $A$ has real entries. 
    Using the terminology of \cite{MM17:DimensionKronecker}, our \Cref{lemma:upperbound_dimensions} can be rephrased as follows. 
    \begin{corollary}
    \label{cor:lowerbound_expfamily}
        Let $\mathcal{E}_A$ be the exponential family associated with the matrix $A \in \bbZ^{(n+1)\times(N+1)}$ and let $\mathcal{M}_r(\mathcal{E})$ denote its $r$th mixture model. 
        Then, the dimension of the $R$th mixture model $\mathcal{M}_R(X)$, with $R=\sum_k(r_k-1)+1$, is a lower bound for the dimension of the Hadamard product of mixture models $\mathcal{M}_{r_1}(\mathcal{E}_A)\star \cdots \star \mathcal{M}_{r_m}(\mathcal{E}_A)$.
    \end{corollary} 
\end{remark} 

Before getting into the technicalities of the proof, we present our strategy: 
\begin{enumerate}
    \item we compute the Jacobian of the parametrization $\widehat\Phi_{A \otimes \bar{B}_{\bfr'}} = \widehat\pi_{\bfr} \circ \widehat\varphi_{A \otimes \bar{B}_{\bfr'}}$ of the affine cone of $\sigma_\bfr(X)$, see \eqref{eq:reparametrize_hadamard};
    \item we perform a modification to the Jacobian of $\widehat\Phi_{A \otimes \bar{B}_{\bfr'}}$ that preserves its rank; 
    \item for a general choice of parameters ${\bf Y} = (\bfy_1,\ldots,\bfy_R)$, we construct a $1$-parameter family of modified Jacobian matrices of $\widehat\Phi_{A \otimes \bar{B}_{\bfr'}}$  evaluated at $\bfY_\nu = (\bfy^\nu_0,\bfy^\nu_{k,j})_{\substack{k=1,\ldots,m \\ j = 1,\ldots,r_k}}$ whose limit for $\nu \to 0$ tends to the Jacobian matrix of the parametrization of the affine cone of $\sigma_R(X)$ given by $\widehat\Phi_{A \otimes \bar{I}_{R}}=\widehat\pi_R \circ \widehat\varphi_{A \otimes \bar{I}_R}$, see \eqref{eq:reparametrization_secant}, evaluated at the points $\bfY$;
    \item since the rank of matrices is lower semicontinuous and the parameters $\bfY$ have been chosen generically, this concludes the proof.
\end{enumerate} 
Note that point $(3)$ exploits the fact that the matrix $\bar{B}_{\bfr'}$ contains, as submatrix, the matrix $\bar{I}_R$. Indeed, the parametrization of the affine cone of $\sigma_R(X)$ given by $\widehat\Phi_{A \otimes \bar{I}_{R}} = \widehat\pi_{R} \circ \widehat\varphi_{A \otimes \bar{I}_{R}}$ in \eqref{eq:reparametrization_secant} is a truncation of the parametrization of the affine cone of $\sigma_\bfr(X)$ given by $\widehat\Phi_{A \otimes \bar{B}_{\bfr'}} = \widehat\pi_{\bfr} \circ \widehat\varphi_{A \otimes \bar{B}_{\bfr'}}$ in \eqref{eq:reparametrize_hadamard}, up to higher degree terms.

For convenience of the reader, to illustrate the heavy notations, \Cref{example:running_3} presents a particular instance of the proof looking at the first previously unknown case of $\sigma_{(2,3)}(X)$ for $X$ being the Rational Normal Curve of~$\bbP^8$.

\begin{proof}[Proof of \Cref{lemma:upperbound_dimensions}]
    First of all, we consider the parametrization of $X$ given by a matrix $\bar{A}\in \bbZ^{(n+1)\times (N+1)}$ which has the same row-span of $A$, but the first row is equal to $\mathds{1}$ and the first column is equal to $(1,0,\ldots,0)$, see \Cref{rmk:reparametrize_toric}.
    
    The (projective) dimension of $\sigma_\bfr(X)$ is \begin{equation}\label{eq:dimension_rank_jacobian}
        \dim \sigma_\bfr(X) = \max_{\bfY\in ((\bbC^\times)^{n+1})^{\times R}}~\rk{\rm Jac}(\widehat\Phi_{\bar{A} \otimes \bar{B}_{\bfr'}})(\bfY)-1.
    \end{equation}
    Recall the notation $\bfr' = \bfr - \mathds{1}$ and $[r] = \{0,\ldots,r\}$. For $\ell \in \{0,\ldots,n\}$, we denote by $\bar{A}^{(\ell)}$ the matrix $\bar{A}\in \bbZ^{(n+1)\times (N+1)}$ where the $\ell$th row has all positive entries decreased by one.
    
    The Jacobian of $\widehat\Phi_{\bar{A} \otimes \bar{B}_{\bfr'}}$ evaluated at the parameters \[\bfY = 
    (\bfy_0 |\bfy_{1,1}|\cdots | \bfy_{1,r_{1}-1} | \cdots | \bfy_{m,1}|\cdots | \bfy_{m,r_{m}-1})
    \in (\bbC^\times)^{(n+1) \times R}\] is a matrix of size $R(n+1) \times (N+1)$ whose $h$th column, for $h \in \{0,\ldots,N\}$, has entries: 
    \begin{itemize}
        \item for $\ell = 0,\ldots n$,
        \[\partial_{y_{0,\ell}}(\widehat\Phi_{A \otimes \bar{B}_{\bfr'}})_h(\bfY) = \alpha_{\ell,h} \cdot \widehat\varphi_{\bar{A}^{(\ell)}}(\bfy_0)_h \star \sum_{\bfj \in [r'_1]\times\cdots\times [r'_m]} 
        \widehat\varphi_{\bar{A}}(\bfy_\bfj)_h \]
        \item for $k = 1,\ldots,m, j = 1,\ldots,r_k-1, \ell = 0,\ldots n$,
        \[\partial_{y_{k,j,\ell}}(\widehat\Phi_{\bar{A} \otimes \bar{B}_{\bfr'}})_h(\bfY)  = \alpha_{\ell,h}\cdot \widehat\varphi_{\bar{A}^{(\ell)}}(\bfy_{k,j})_h\star \sum_{\substack{\bfj \in [r'_1]\times \cdots \times [r'_m] \\ \text{ s.t. } j_k = 0}}  
            \widehat\varphi_{\bar{A}}(\bfy_0 \star \bfy_\bfj)_h.\]
    \end{itemize}
    Indeed, since $\widehat\varphi_{\bar{A}}$ is a monomial map, $\partial_{x_\ell}\widehat\varphi_{\bar{A}}(\bfx)_h = \alpha_{\ell,h}\cdot \widehat\varphi_{\bar{A}^{(\ell)}}(\bfx)_h$.
    
    Now, we consider a modification of ${\rm Jac}(\widehat\Phi_{\bar{A} \otimes \bar{B}_{\bfr'}})(\bfY)$ where we multiply each row by the 
    coordinate with respect to which the derivative was taken:
    \[
        {\rm diag}(\bfy_{0},\ldots,\bfy_{k,j},\ldots) \cdot {\rm Jac}(\widehat\Phi_{\bar{A} \otimes \bar{B}_{\bfr'}})(\bfY) =: K_{\bar{A} \otimes \bar{B}_{\bfr'}}(\bfY).
    \]
    Since we are multiplying the rows by non-zero scalars, the rank is preserved. 
    Note that 
    \begin{align*}
        K_{\bar{A}\otimes \bar{B}_{\bfr'}}(\bfY) & 
        = \eta_{\bar{A},\bar{B}_{\bfr'}}(\bfY) \odot \bar{A} 
        \in \bbC^{(n+1)R \times (N+1)},
    \end{align*}
    where $\odot$ denotes the Khatri-Rao product (i.e., the column-wise Kronecker product), and 
    \begin{equation}
        {\eta_{\bar{A},\bar{B}_{\bfr'}}(\bfY)}
        = \begin{pNiceMatrix}[first-col]
            0 & \sum_{\bfj \in [r'_1]\times\cdots\times [r'_m]} \widehat\varphi_{\bar{A}}(\bfy_0 \star \bfy_\bfj) \\
            & \vdots \\
            (k,j) & \sum_{\substack{\bfj \in [r'_1]\times\cdots\times [r'_m] \\
            \text{ s.t. }j_k = j \geq 1}} \widehat\varphi_{\bar{A}}(\bfy_0 \star \bfy_\bfj) \\
            & \vdots
        \end{pNiceMatrix}_{\substack{k=1,\ldots,m \\ j = 1,\ldots,r_k-1}} \in
        {\bbC^{R \times (N+1)}} . 
        \label{eq:etaB}
\end{equation} 
    A similar construction can be done for the parametrization $\widehat\Phi_{\bar{A} \otimes \bar{I}_R}$ of the affine cone of the $R$th secant variety.
    In this case, the $h$th column of the Jacobian of $\widehat\Phi_{\bar{A} \otimes \bar{I}_R}$ evaluated at $\bfY = (
    {\bfy_1}|\cdots|
    \bfy_R) \in (\bbC^\times)^{(n+1) \times R}$ has entries, for $h \in \{0,\ldots,N\}$,
    \begin{itemize}
    \item for $\ell = 0,\ldots, n$,
        \[\partial_{y_{1,\ell}}(\widehat\Phi_{\bar{A} \otimes \bar{I}_{R}})_h(\bfY) = \alpha_{\ell,h}
        \cdot \widehat\varphi_{\bar{A}^{(\ell)}}(\bfy_1)_h \star \left(\mathds{1}+ \sum_{j=2}^R\widehat\varphi_{\bar{A}}(\bfy_{j})_h\right);\]
        \item for $j = {2},\ldots,R, \ell = 0,\ldots, n$,
        \[
        \partial_{y_{j,\ell}}(\widehat\Phi_{\bar{A} \otimes \bar{I}_{R}})_h(\bfY) = \alpha_{\ell,h} \cdot
        \widehat\varphi_{\bar{A}}(\bfy_1)_h\star\widehat\varphi_{\bar{A}^{(\ell)}}(\bfy_j)_h 
        \]
    \end{itemize}
    Again, we multiply each row by the coordinate with respect to which the derivative is being taken,
    \[
        {\rm diag}(\bfy_{1},\ldots,\bfy_R) \cdot {\rm Jac}(\widehat\Phi_{\bar{A} \otimes \bar{I}_{R}})(\bfY) =: K_{\bar{A} \otimes \bar{I}_{R}}(\bfY).
    \]
    Since we are multiplying by non-zero coefficients, the rank is preserved. 
    Note that 
    \[
        K_{\bar{A} \otimes \bar{I}_{R}}(\bfY) = 
        \eta_{\bar{A} , \bar{I}_{R}}(\bfY) \odot \bar{A}
        \in \bbC^{(n+1)R \times (N+1)},
    \]
    where
    \begin{equation} 
        \eta_{\bar{A} , \bar{I}_{R}}(\bfY) = 
        \begin{pNiceMatrix}[first-col]
            1 & \widehat\varphi_{\bar{A}}(\bfy_1) + \sum_{j=2}^R \widehat\varphi_{\bar{A}}( 
            {\bfy_1}\star \bfy_j) \\
            & \vdots \\
            j & \widehat\varphi_{\bar{A}}( 
            {\bfy_1}\star \bfy_j) \\
            & \vdots
        \end{pNiceMatrix}_{j = 2,\ldots,R} \in 
        {\bbC^{R \times (N+1)}} . 
        \label{eq:etaI}
   \end{equation} 
    Now, we are ready to prove the following crucial claim.
    
    \begin{claim}\label{claim:1}
        Given a general choice $\bfY\in ((\bbC^\times)^{n+1})^R$, there exist a $1$-parameter family $(\bfY_\nu)_{\nu \in \bbR}\in ((\bbC^\times)^{n+1})^R$ and matrices $L(\bfY_\nu) \in (\bbC^\times)^{R(n+1)\times R(n+1)}, R(\bfY_\nu) \in (\bbC^\times)^{(N+1)\times (N+1)}$ such that 
        \[
            \lim_{\nu \to 0} L(\bfY_\nu)K_{\bar{A}\otimes\bar{B}_{\bfr'}}(\bfY_\nu)R(\bfY_\nu) = \bar K(\bfY), \quad \text{with}\quad \operatorname{rowspan}(\bar K(\bfY)) = \operatorname{rowspan}(K_{\bar{A} \otimes \bar{I}_R}(\bfY)).
        \]
    \end{claim}
    Before proving the claim, let us see why this is enough to conclude our proof. 
    By \Cref{eq:dimension_rank_jacobian}, 
    \begin{equation}\label{eq:chain_1}
        \dim \sigma_{\bfr}(X)+1 \geq \rk {\rm Jac}(\widehat\Phi_{\bar{A} \otimes \bar{B}_{\bfr'}})(\bfY_\nu); 
    \end{equation}
    and, by construction, 
    \begin{equation}\label{eq:chain_2}
        \rk {\rm Jac}(\widehat\Phi_{\bar{A}\otimes \bar{B}_{\bfr'}})(\bfY_\nu) = \rk K_{\bar{A}\otimes \bar{B}_{\bfr'}}(\bfY_\nu). 
    \end{equation}
    Multiplying by matrices from the left and from the right can only decrease the rank, and thus, for any $\nu$,
    \begin{equation}\label{eq:chain_3}
        \rk K_{\bar{A} \otimes \bar{B}_{\bfr'}}(\bfY_\nu) \geq \rk\left(L(\bfY_\nu)K_{\bar{A}\otimes\bar{B}_{\bfr'}}(\bfY_\nu)R(\bfY_\nu)\right).
    \end{equation}
   Then, due to the semicontinuity of the rank, \Cref{claim:1} allows us to conclude that, for $\nu$ sufficiently small, 
    \begin{equation}\label{eq:chain_4}
        \rk \left(L(\bfY_\nu)K_{\bar{A}\otimes\bar{B}_{\bfr'}}(\bfY_\nu)R(\bfY_\nu)\right) \geq \rk\bar{K}(\bfY) = \rk K_{\bar{A} \otimes \bar{I}_R}(\bfY).
    \end{equation}
    Finally, by construction of the matrix $K_{\bar{A} \otimes \bar{I}_R}(\bfY)$ and by generality of $\bfY$, 
    \begin{equation}\label{eq:chain_5}
        \rk K_{\bar{A} \otimes \bar{I}_R}(\bfY) = \rk {\rm Jac}(\widehat\Phi_{\bar{A} \otimes \bar{I}_R})(\bfY) = \dim \sigma_R(X)+1.
    \end{equation}
    Putting all together \Cref{eq:chain_1,eq:chain_2,eq:chain_3,eq:chain_4,eq:chain_5}, the proof of the theorem is concluded.

    \begin{proof}[Proof of \Cref{claim:1}]\let\qed\relax
        Without loss of generality, up to the action of the torus $((\bbC^\times)^{n+1})^R$, we may assume that $\bfy_1 = \mathds{1}$, i.e., $\bfY = (\mathds{1}|\bfy_2|\cdots|\bfy_R)$. Define $\bfY_\nu$ by multiplying the first entry of each column of $\bfY$ by {$\nu$,} namely, 
        \[
        \bfY_\nu = \bfY \star \begin{pmatrix}
        {\nu} 
        \mathds{1} \\ \mathds{1} \\ \vdots \\  \mathds{1} \end{pmatrix} =: (\bfy^\nu_0 | \cdots | \bfy^\nu_{k,j} | \cdots)_{\substack{k = 1,\ldots,m \\ j = 1,\ldots,r_k-1}}.
        \]  
        For $\nu \neq 0$, let $\lambda(\nu) =
        {\operatorname{diag}(1,\nu^{-1},\ldots,\nu^{-1})\in (\bbC^\times)^{R \times R}}$ 
        and 
        \[
            L(\nu) = \lambda(\nu)\otimes I_{n+1}  \in (\bbC^{\times})^{(n+1)R\times (n+1)R}.
        \]
        Set 
        \[
            R(\bfY_\nu) = 
            \operatorname{diag}\left(\sum_{\bfj \in [r'_1]\times\cdots\times [r'_m]} \widehat\varphi_{A}(\bfy^\nu_0 \star \bfy^\nu_\bfj)\right)^{-1}{\in\bbC^{(N+1)\times(N+1)}} .
        \]
        Consider the matrix defined in \Cref{eq:etaB} and observe that, since $\bar{A}$ has the first row equal to $\mathds{1}$ and by construction of $\bfY_\nu$, the rows of $\eta_{\bar{A}, \bar{B}_{\bfr'}}(\bfY_\nu)$ take the following form:  
    \begin{itemize}
    \item for $k = 0$:
        \begin{equation}\label{eq:expanded_formulas}\sum_{\bfj \in [r'_1]\times\cdots\times [r'_m]} \widehat\varphi_{A}(\bfy^\nu_0 \star \bfy^\nu_\bfj) =
        \nu \widehat\varphi_A(\bfy_1) + 
        \nu^2 \sum_{j=2}^R\widehat\varphi_A(\bfy_1\star \bfy_j) +
        \nu^3 (\text{ other terms })\end{equation}
        \item for $k = 1,\ldots,m, j = 1,\ldots,r_k-1$,
        \[\sum_{\substack{\bfj \in [r'_1]\times\cdots\times [r'_m] \\
            \text{ s.t. }j_k = j \geq 1}} \widehat\varphi_{A}(\bfy_0^\nu \star \bfy_\bfj^\nu) =
            \nu^2 \widehat\varphi_A(\bfy_1\star \bfy_{i_{k,j}}) +
            \nu^3 (\text{ other terms }).\]
            where $i_{k,j} = 1+\sum_{h=1}^{k-1} (r_h - 1)+j.$
    \end{itemize}
    Now, we consider 
    $\lambda(\nu)\eta_{\bar{A}, \bar{B}_{\bfr'}}(\bfY_\nu) R(\bfY_\nu) \in \bbC^{R\times (N+1)}$. Observe that, by construction of $\bfY_\nu$, the $h$th column, for $h\in\{0,1\ldots, N\}$, has the $i$th entry, $i \in \{1,\ldots,R\}$, equal to: 
    \begin{itemize}
    \item for $i = 1$:
      \[
        1 \cdot \sum_{\bfj \in [r_1]\times\cdots\times [r_m]} \widehat\varphi_{A}(\bfy^\nu_0 \star \bfy^\nu_\bfj)_h \cdot \frac{1}{\sum_{\bfj \in [r_1']\times\cdots\times [r_m']} \widehat\varphi_{A}(\bfy^\nu_0 \star \bfy^\nu_\bfj)_h} = 1 ;
        \]
    \item{for $i = 2,\ldots,R$:}
      \begin{align*}\nu^{-1} \cdot \sum_{\substack{\bfj \in [r_1']\times\cdots\times [r_m'] \\
            \text{ s.t. }j_k \geq 1}} \widehat\varphi_{A}(\bfy_0^\nu \star \bfy_\bfj^\nu)_h &\cdot  \frac{1}{\sum_{\bfj \in [r_1']\times\cdots\times [r_m']}\widehat \varphi_{A}(\bfy^\nu_0 \star \bfy^\nu_\bfj)_h} 
            \\ 
            =& \frac{\nu \widehat\varphi_A(\bfy_1\star \bfy_i)_h + \nu^2 (\text{ other terms })}{\nu \widehat\varphi_A(\bfy_1)_h + \nu^2 \sum_{j=2}^R\widehat\varphi_A(\bfy_1\star \bfy_j)_h + \nu^3(\text{ other terms }) } \\
            =& 
            \frac{
    \widehat\varphi_A(\bfy_i)_h + \nu (\text{ other terms })}{ 1 + \nu \sum_{j=2}^R\widehat\varphi_A( \bfy_j)_h + \nu^2(\text{ other terms }) }
    \end{align*}
    where the latter equality uses \Cref{eq:expanded_formulas} as well as $\bfy_1 = \mathds{1}$ and $\varphi_A(\mathds{1}) = \mathds{1}$. 
    \end{itemize}
Then 
    \begin{equation}\label{eq:limit}
        \lim_{\nu \to 0} \left(\lambda(\nu)\eta_{A, \bar{B}_{\bfr'}}(\bfY_\nu) 
        R(\bfY_\nu)\right) 
        = 
        \begin{pmatrix}
        \mathds{1} \\
        \vdots\\
        \widehat\varphi_A(\bfy_j)\\
        \vdots
        \end{pmatrix}
=: 
        \bar\eta(\bfY)
    \end{equation}
Comparing with \Cref{eq:etaI} the row span of $\bar \eta(\bfY)$ is equal to the row span of $\eta_{A, \bar I_R}(\bfY)$. 
    
    Since 
    \[
        L(\nu)K_{A\otimes \bar{B}_{\bfr'}}(\bfY_\nu)
        R(\bfY_\nu) = 
(\lambda(\nu)\eta_{A, \bar{B}_{\bfr'}}(\bfY_\nu) 
        R(\bfY_\nu)) 
        \odot A 
        \quad \text{ and }\quad \bar K (\bfY) = \bar \eta(\bfY) \odot A,
    \]
    \Cref{eq:limit} concludes the proof.  
    \end{proof}
    \end{proof}

We illustrate the proof in the case of our running example. See \Cref{example:running_1,example:running_2}.

\begin{example}\label{example:running_3}
    Let $X\subset \bbP^8$ be the rational normal curve defined as the image of the monomial map associated to the following $(n+1)\times(N+1)$-matrix with $(n+1)=2$ and $(N+1)=9$:
\[
    A = \begin{pmatrix}
        8 & 7 & \cdots & 1 & 0 \\
        0 & 1 & \cdots & 7 & 8
    \end{pmatrix} \in \bbZ^{2 \times 9},
\]
namely, $X$ is the Zariski closure of the image of the monomial map $\varphi_A:\bfz = (z_0,z_1)\in (\bbC^\times)^2\mapsto (\bfz^{(8,0)}:\bfz^{(7,1)}:\cdots:\bfz^{(0,8)})=(z_0^8:z_0^7z_1:\cdots:z_1^8)\in \bbP^8$. As commented in \Cref{rmk:reparametrize_toric}, we can reparametrize (an open dense subset of) the affine cone of the rational normal curve by using the matrix
\[
    \bar{A} = \begin{pmatrix}
        1 & 1 & \cdots & 1 \\
        0 & 1 & \cdots & 8
    \end{pmatrix} \in \bbZ^{2\times 9} , 
\]
namely, $\widehat\varphi_{\bar{A}} : \bfy \in(\bbC^\times)^2 \mapsto (\bfy^{(1,0)},\bfy^{(1,1)},\ldots,\bfy^{(1,8)})\in(\bbC^\times)^9$. The image of $\widehat{\varphi}_{\bar{A}}$, which is a cone in $(\bbC^\times)^9$, is a subset of the affine chart $U_0 \subset \bbP^8$ given by the first coordinate different from zero. The two parametrizations $\varphi_A$ and $\varphi_{\bar{A}}$ coincide on the corresponding affine open chart of the rational normal curve. We consider $\sigma_{(2,3)}(X) = \sigma_2(X)\star \sigma_3(X)$ and we want to show that $\dim \sigma_{(2,3)}(X)\geq \dim \sigma_4(X).$ We compute the Jacobians of the two parametrizations described in \Cref{example:running_1,example:running_2}:
\begin{itemize}
    \item ${\rm Jac}~\widehat\Phi_{{\bar{A}} \otimes \bar{B}_{(1,2)}}(\bfY)$ has the $h$th column, $h \in \{0,1,\ldots, 8\}$, with entries:
    \begin{align*}
        \partial_{y_{0,0}} & : & \bfy_{0}^{(0,h)} \star \left[{1} + \bfy_{1,1}^{(1,h)} + \bfy_{2,1}^{(1,h)} + \bfy_{2,2}^{(1,h)} + (\bfy_{1,1} \star \bfy_{2,1})^{(1,h)} + (\bfy_{1,1} \star \bfy_{2,2})^{(1,h)}\right] \\
        \partial_{y_{0,1}} & : & h\bfy_{0}^{(1,h-1)} \star \left[{1} + \bfy_{1,1}^{(1,h)} + \bfy_{2,1}^{(1,h)} + \bfy_{2,2}^{(1,h)} + (\bfy_{1,1} \star \bfy_{2,1})^{(1,h)} + (\bfy_{1,1} \star \bfy_{2,2})^{(1,h)}\right] \\
        \partial_{y_{1,1,0}} & : & \bfy_{1,1}^{(0,h)} \star \left[\bfy_{0}^{(1,h)} + (\bfy_{0} \star \bfy_{2,1})^{(1,h)} + (\bfy_{0} \star \bfy_{2,2})^{(1,h)}\right] \\
        \partial_{y_{1,1,1}} & : & h\bfy_{1,1}^{(1,h-1)} \star \left[\bfy_{0}^{(1,h)} + (\bfy_{0} \star \bfy_{2,1})^{(1,h)} + (\bfy_{0} \star \bfy_{2,2})^{(1,h)}\right] \\
        \partial_{y_{2,1,0}} & : & \bfy_{2,1}^{(0,h)} \star \left[\bfy_{0}^{(1,h)} + (\bfy_{0} \star \bfy_{1,1})^{(1,h)} \right] \\
        \partial_{y_{2,1,1}} & : & h\bfy_{2,1}^{(1,h-1)} \star \left[\bfy_{0}^{(1,h)} + (\bfy_{0} \star \bfy_{1,1})^{(1,h)} \right] \\ 
        \partial_{y_{2,2,0}} & : & \bfy_{2,2}^{(0,h)} \star \left[\bfy_{0}^{(1,h)} + (\bfy_{0} \star \bfy_{1,1})^{(1,h)} \right] \\
        \partial_{y_{2,2,1}} & : & h\bfy_{2,2}^{(1,h-1)} \star \left[\bfy_{0}^{(1,h)} + (\bfy_{0} \star \bfy_{1,1})^{(1,h)} \right] \\ 
    \end{align*}
    \item  ${\rm Jac}~\widehat\Phi_{{\bar{A}} \otimes \bar{I}_{4}}(\bfY)$ has the $h$th column, $h \in \{0,1,\ldots,8\}$, with entries:
    \begin{align*}
        \partial_{y_{1,0}} & : & \bfy_{1}^{(0,h)} \star \left[{1} + \bfy_2^{(1,h)} + \bfy_{3}^{(1,h)} + \bfy_{4}^{(1,h)}\right] \\
        \partial_{y_{1,1}} & : & h\bfy_{1}^{(1,h-1)} \star \left[{1} + \bfy_2^{(1,h)} + \bfy_{3}^{(1,h)} + \bfy_{4}^{(1,h)}\right] \\
        \partial_{y_{2,0}} & : & \bfy_{1}^{(1,h)}\star \bfy_{2}^{(0,h)}  \\
        \partial_{y_{2,1}} & : & h\bfy_{1}^{(1,h)}\star \bfy_{2}^{(1,h-1)}  \\
        \partial_{y_{3,0}} & : & \bfy_{1}^{(1,h)}\star \bfy_{3}^{(0,h)}  \\
        \partial_{y_{3,1}} & : & h\bfy_{1}^{(1,h)}\star \bfy_{3}^{(1,h-1)}  \\
        \partial_{y_{4,0}} & : & \bfy_{1}^{(1,h)}\star \bfy_{4}^{(0,h)}  \\
        \partial_{y_{4,1}} & : & h\bfy_{1}^{(1,h)}\star \bfy_{4}^{(1,h-1)}  \\
        \end{align*}
\end{itemize}
Actually, instead of the Jacobians, we consider the following matrices
\begin{itemize}
    \item $K_{{\bar{A}}\otimes \bar{B}_{(1,2)}}(\bfY):= \diag(\bfy_0,\bfy_{1,1},\bfy_{2,1},\bfy_{2,2}) \cdot {\rm Jac}~\widehat\Phi_{{\bar{A}}\otimes \bar{B}_{(1,2)}}(\bfY) \in \bbC^{8\times 9}$;
    \item $K_{{\bar{A}}\otimes \bar{I}_{4}}(\bfY) := \diag(\bfy_1,\bfy_2,\bfy_3,\bfy_4) \cdot {\rm Jac}~\widehat\Phi_{\bar{A}\otimes \bar{I}_{4}}(\bfY)\in \bbC^{8\times 9}$;
\end{itemize}
which can be conveniently described as the Kathri-Rao products
\[
    K_{{\bar{A}}\otimes \bar{B}_{(1,2)}}(\bfY) = \eta_{\bar{A}, \bar{B}_{(1,2)}}(\bfY) \odot \bar{A} \quad \text{ and } \quad K_{{\bar{A}}\otimes \bar{I}_4}(\bfY) = \eta_{\bar{A}, \bar{I}_{4}}(\bfY) \odot \bar{A},
\]
where
\begin{itemize}
    \item $\eta_{{\bar{A}}\otimes \bar{B}_{(1,2)}}(\bfy) \in \bbC^{4\times 9}$ is the matrix whose $h$th column, for $h\in \{0,\ldots,8\}$, is 
    \[
    \left(\begin{array}{l}
        \bfy_{0}^{(1,h)} + (\bfy_0\star \bfy_{1,1})^{(1,h)} + (\bfy_0\star \bfy_{2,1})^{(1,h)} + (\bfy_0\star \bfy_{2,2})^{(1,h)} + (\bfy_0\star\bfy_{1,1} \star \bfy_{2,1})^{(1,h)} + (\bfy_0\star\bfy_{1,1} \star \bfy_{2,2})^{(1,h)} \\
        (\bfy_0 \star \bfy_{1,1})^{(1,h)} + (\bfy_0 \star \bfy_{1,1} \star \bfy_{2,1})^{(1,h)} + (\bfy_0 \star \bfy_{1,1} \star \bfy_{2,2})^{(1,h)} \\
        (\bfy_0 \star \bfy_{2,1})^{(1,h)} + (\bfy_0 \star \bfy_{1,1} \star \bfy_{2,1})^{(1,h)}  \\
        (\bfy_0 \star \bfy_{2,2})^{(1,h)} + (\bfy_0 \star \bfy_{1,1} \star \bfy_{2,2})^{(1,h)} 
    \end{array}\right);\]
    \item $\eta_{{\bar{A}}\otimes \bar{I}_{4}}(\bfy)\in \bbC^{4\times 9}$ is the matrix whose $h$th column, for $h\in \{0,\ldots,8\}$, is  
    \[\left(\begin{array}{l}
        \bfy_{1}^{(1,h)} + (\bfy_1\star \bfy_{2})^{(1,h)} + (\bfy_1\star \bfy_{3})^{(1,h)} + (\bfy_1\star \bfy_{4})^{(1,h)} \\
        (\bfy_1 \star \bfy_{2})^{(1,h)} \\
        (\bfy_1 \star \bfy_{3})^{(1,h)} \\
        (\bfy_1 \star \bfy_{4})^{(1,h)}
    \end{array}\right). 
    \]
\end{itemize}
Fix generic $\bfY = (\bfy_1|\bfy_2|\bfy_3|\bfy_4) \in (\bbC^\times)^{2\times 4}$: without loss of generality we may assume $\bfy_1 = \mathds{1}$. For any $\nu$, define $\bfY_\nu = (\bfy_0^\nu|\bfy_{1,1}^\nu|\bfy_{2,1}^\nu|\bfy_{2,2}^\nu)\in (\bbC^\times)^{2 \times 9}$ where
\begin{align*}
    \bfy_0^\nu & = \left(\nu,1\right)\star \bfy_1 \in (\bbC^\times)^2,\\
    \bfy_{1,1}^\nu & = \left(\nu,1\right)\star \bfy_2 \in (\bbC^\times)^2,\\
    \bfy_{2,1}^\nu & = \left(\nu,1\right)\star \bfy_3 \in (\bbC^\times)^2,\\
    \bfy_{2,2}^\nu & = \left(\nu,1\right)\star \bfy_4 \in (\bbC^\times)^2.
\end{align*}
Now, for any $\nu\neq 0$, we consider
\[
    \lambda(\nu)=\diag(1,\nu^{-1},\nu^{-1},\nu^{-1})\in (\bbC^\times)^{4 \times 4}
\]
and define
\[
    R(\bfY_\nu) = \diag\left(\ldots,\frac{1}{ \substack{(\bfy^\nu_{0})^{(1,h)} + (\bfy^\nu_0\star \bfy^\nu_{1,1})^{(1,h)} + (\bfy^\nu_0\star \bfy^\nu_{2,1})^{(1,h)} + (\bfy^\nu_0\star \bfy^\nu_{2,2})^{(1,h)}+ (\bfy^\nu_0\star\bfy^\nu_{1,1} \star \bfy^\nu_{2,1})^{(1,h)} + (\bfy^\nu_0\star\bfy^\nu_{1,1} \star \bfy^\nu_{2,2})^{(1,h)}}},\ldots\right)_{h=0,\ldots,8}.
\]
We show the proof of \Cref{claim:1} for this specific case. Consider the matrix $\eta_{\bar{A},\bar{B}_{1,2}}(\bfY_\nu)\in (\bbC^\times)^{4\times 9}$ whose $h$th column, for $h \in \{0,\ldots,8\}$, is, by construction of $\bfY_\nu$,
\[
    \left(\begin{array}{l}
        \nu\bfy_{1}^{(1,h)} + \nu^{2}\left((\bfy_1\star \bfy_{2})^{(1,h)} +(\bfy_1\star \bfy_{3})^{(1,h)} + (\bfy_1\star \bfy_{4})^{(1,h)}\right) + \nu^{3}\left((\bfy_1\star\bfy_{2} \star \bfy_{3})^{(1,h)} + (\bfy_1\star\bfy_{2} \star \bfy_{4})^{(1,h)}\right) \\
        \nu^{2}(\bfy_1 \star \bfy_{2})^{(1,h)} + \nu^{3}\left(( 
        \bfy_1
        \star \bfy_{2} \star \bfy_{3})^{(1,h)} + (\bfy_1 \star \bfy_{2} \star \bfy_{4})^{(1,h)}\right) \\
        \nu^{2}(\bfy_1 \star \bfy_{3})^{(1,h)} + \nu^{3}(\bfy_1 \star \bfy_{2} \star \bfy_{3})^{(1,h)}  \\
        \nu^{2}(\bfy_1 \star \bfy_{4})^{(1,h)} + \nu^{3}(\bfy_1 \star \bfy_{2} \star \bfy_{4})^{(1,h)} 
    \end{array}\right).
\]
Now:
\begin{itemize}
    \item we first multiply the latter matrix on the left by $\lambda(\nu)$ to get:
    \[
    \left(\begin{array}{l}
        \nu\bfy_{1}^{(1,h)} + \nu^{2}\left((\bfy_1\star \bfy_{2})^{(1,h)} +(\bfy_1\star \bfy_{3})^{(1,h)} + (\bfy_1\star \bfy_{4})^{(1,h)}\right) + \nu^{3}\left((\bfy_1\star\bfy_{2} \star \bfy_{3})^{(1,h)} + (\bfy_1\star\bfy_{2} \star \bfy_{4})^{(1,h)}\right) \\
        \nu(\bfy_1 \star \bfy_{2})^{(1,h)} + \nu^{2}\left((
        \bfy_1
        \star \bfy_{2} \star \bfy_{3})^{(1,h)} + (\bfy_1 \star \bfy_{2} \star \bfy_{4})^{(1,h)}\right) \\
        \nu(\bfy_1 \star \bfy_{3})^{(1,h)} + \nu^{2}(\bfy_1 \star \bfy_{2} \star \bfy_{3})^{(1,h)}  \\
        \nu(\bfy_1 \star \bfy_{4})^{(1,h)} + \nu^{2}(\bfy_1 \star \bfy_{2} \star \bfy_{4})^{(1,h)} 
    \end{array}\right);
\]
    \item then, we multiply on the right by $R(\bfY_\nu)$ and we get:
    \[
    \begin{pmatrix}
        1 \\
        \frac{\nu(\bfy_1 \star \bfy_{2})^{(1,h)} + \nu^{2}\left((
        {\bfy_1} 
        \star \bfy_{2} \star \bfy_{3})^{(1,h)} + (\bfy_1 \star \bfy_{2} \star \bfy_{4})^{(1,h)}\right)}{{{\nu\bfy_{1}^{(1,h)} + \nu^{2}\left((\bfy_1\star \bfy_{2})^{(1,h)} + (\bfy_1\star \bfy_{3})^{(1,h)} + (\bfy_1\star \bfy_{4})^{(1,h)}\right)+ \nu^{3}\left((\bfy_1\star\bfy_{2} \star \bfy_{3})^{(1,h)} + (\bfy_1\star\bfy_{2} \star \bfy_{4})^{(1,h)}\right)}}} \\
        \frac{\nu(\bfy_1 \star \bfy_{3})^{(1,h)} + \nu^{2}(\bfy_1 \star \bfy_{2} \star \bfy_{3})^{(1,h)}}{{{\nu\bfy_{1}^{(1,h)} + \nu^{2}\left((\bfy_1\star \bfy_{2})^{(1,h)} + (\bfy_1\star \bfy_{3})^{(1,h)} + (\bfy_1\star \bfy_{4})^{(1,h)}\right)+ \nu^{3}\left((\bfy_1\star\bfy_{2} \star \bfy_{3})^{(1,h)} + (\bfy_1\star\bfy_{2} \star \bfy_{4})^{(1,h)}\right)}}} \\
        \frac{\nu(\bfy_1 \star \bfy_{4})^{(1,h)} + \nu^{2}(\bfy_1 \star \bfy_{2} \star \bfy_{4})^{(1,h)}}{{{\nu\bfy_{1}^{(1,h)} +  \nu^{2}\left((\bfy_1\star \bfy_{2})^{(1,h)} + (\bfy_1\star \bfy_{3})^{(1,h)} + (\bfy_1\star \bfy_{4})^{(1,h)}\right)+ \nu^{3}\left((\bfy_1\star\bfy_{2} \star \bfy_{3})^{(1,h)} + (\bfy_1\star\bfy_{2} \star \bfy_{4})^{(1,h)}\right)}}}
    \end{pmatrix};
\]
    \item by simplifying the $\nu$'s and passing to the limit $\nu \rightarrow 0$, since $\bfy_1=\mathds{1}$, we get the matrix of size $4\times 9$ whose $h$th column, for $h\in \{0,\ldots,8\}$, is
    \[
    \lim_{\nu\rightarrow 0} \left(\lambda(\nu)\eta_{\bar{A},\bar{B}_{(1,2)}}(\bfY^\nu)R(\bfY^\nu)\right)_h = \begin{pmatrix}
        1 \\
        \frac{(\bfy_1\star \bfy_2)^{(1,h)}}{\bfy_1^{(1,h)}} \\
        \frac{(\bfy_1\star \bfy_3)^{(1,h)}}{\bfy_1^{(1,h)}} \\
        \frac{(\bfy_1\star \bfy_4)^{(1,h)}}{\bfy_1^{(1,h)}} \\
    \end{pmatrix} = 
    \begin{pmatrix}
        1 \\
        \bfy_2^{(1,h)} \\
        \bfy_3^{(1,h)} \\
        \bfy_4^{(1,h)} \\
    \end{pmatrix} =: \bar{\eta}(\bfY)_h.
    \]
\end{itemize}
Consequently, if $L(\nu) := \lambda(\nu) \otimes I_2 \in (\bbC^\times)^{8 \times 8}$, then 
\[
    L(\nu)K_{\bar{A},\bar{B}_{(1,2)}}(\bfY^\nu)R(\bfY^\nu) = (\lambda(\nu)\eta_{\bar{A},\bar{B}_{(1,2)}}(\bfY^\nu) R(\bfY^\nu)) \odot \bar{A},
\]
and therefore, 
\[
    \lim_{\nu \rightarrow 0}\left( L(\nu)K_{\bar{A},\bar{B}_{(1,2)}}(\bfY^\nu)R(\bfY^\nu) \right) = \bar{\eta}(\bfY)\odot \bar{A}.
\]
Note that the row-span of the matrix $\bar{\eta}(\bfY)\in (\bbC^\times)^{4\times 9}$ coincides with the rowspan of $\eta_{\bar{A}\otimes\bar{I}_{4}}(\bfY)$: indeed, the latter is obtained from the former by substituting the first row with the sum of all rows. Recall we are assuming, without loss of generality, that $\bfy_1 = \mathds{1}$. Similarly, since we consider the Kathri-Rao product, namely, the column-wise Kronecker product, the rowspan of $\bar{K}(\bfY) = \bar{\eta}(\bfY) \odot \bar{A}$ is equal to the rowspan of $K_{\bar{A}\otimes\bar{I}_4}(\bfY) = \eta_{\bar{A},\bar{I}_4}(\bfY)\odot \bar{A}.$ 
\end{example}

\subsection{General Hadamard ranks} \Cref{lemma:upperbound_dimensions} completes the proof of the chain of inequalities from \Cref{eq:chain_of_dimensions} for embedded toric varieties. Under the assumption that $X$ is not $R$-defective, the leftmost and rightmost terms coincide. In particular, $\dim\sigma_R(X)=\dim\sigma_\bfr(X)$. Recall that, in the particular case $\bfr = (r,\ldots,r)$, this allows us to deduce the smallest $m$ for which $\sigma_\bfr(X)$ fills the ambient space, namely, to compute the general $r$th $X$-Hadamard-rank. 

Motivated by applications to tensor decompositions and, in particular, to the study of discrete Restricted Boltzmann Machines defined in \cite{montufar2015discrete}, we focus on the case of Segre-Veronese varieties. Recall that a full classification is not yet known, although several complete results are known when restricting to particular families: we refer to the recent survey \cite{oneto2025ranks} for more details on the current state-of-the-art and an extended list of references. 

A classical general result ensures non-defectivity of $\sigma_r(X)$  under the geometric constraints that no secant variety is a projective cone, for $r$ smaller than a bound depending on the dimension and codimension of $X$; see \cite[Corollary 2.3]{adlandsvik1988varieties}. The result has been extended in \cite{taveira2023nondefectivity} to varieties $X \subset \bbP^N$ that additionally are invariant under some certain group action. In particular, \cite[Corollary 3.5]{taveira2023nondefectivity}, states such an $X$ is non-defective for $r \leq \frac{N}{\dim X}-\dim X$ and the $r$th secant variety fills the ambient space for $r \geq \frac{N}{\dim X}+\dim X$. A slightly stronger version can be found in \cite[Theorem 2.1]{ballico2024non}. Combining \cite[Corollary 3.5]{taveira2023nondefectivity} with our \Cref{lemma:upperbound_dimensions}, we deduce the following for all Segre-Veronese varieties.
\begin{proposition}
    Let $\bfr = (r_1,\ldots,r_m)$. Then, the Hadamard product of secant varieties of the Segre-Veronese variety of rank-one partially symmetric binary tensors $\sigma_\bfr(SV_{\bfd,\bfn})$ are not Hadamard-defective if 
    \[
        \sum_k(r_k-1)+1 \leq \frac{\prod_{i=1}^k{n_i+d_i \choose d_i}}{(n_1+\cdots+n_k)}-(n_1+\cdots+n_k)
    \]
    and fills the ambient space if 
    \[
        \sum_k(r_k-1)+1 \geq \frac{\prod_{i=1}^k{n_i+d_i \choose d_i}}{n_1+\cdots+n_k}+(n_1+\cdots+n_k).
    \]
    In particular, the $r$th generic partially symmetric Hadamard rank in ${\rm Sym}^{d_1}\bbC^{n_1+1}\otimes\cdots\otimes{\rm Sym}^{d_k}\bbC^{n_k+1}$ is at most equal to $\Bigl\lceil \frac{\prod_{i=1}^k{n_i+d_i \choose d_i} -(n_1+\cdots+n_k)}{(r-1)(n_1+\cdots+n_k+1)}\Bigr\rceil$. 
\end{proposition}

\begin{remark}
    For sake of completeness, we recall a general result on non-defectiveness of toric varieties given in \cite[Theorem 2.13]{laface2022secant}. 
    Here, it is proven that if $X\subset \bbP^N$ is a toric variety, then $\sigma_{r}(X)$ is non-defective for $r < \frac{|P|-m}{\dim X +1}$, where the upper bound is given in terms of the lattice points $|P|$ of the polytope defining $X$ and the maximal number $m$ of lattice points of a hyperplane section of $P$.
\end{remark}

As mentioned, we list here a series of results for the most complete families of Segre-Veronese varieties for which a full classification of defective secant varieties is known. 

\subsubsection{Veronese varieties, i.e., the case of symmetric tensors}\label{subsec:dimension_veronese} In \cite{alexander1995polynomial}, Alexander and Hirshowitz completed the classification of defective secant varieties of Veronese varieties $V_{d,n}$ finishing a work started more than 100 years before when a series of defective Veronese varieties were discovered. Recall that Veronese varieties parametrize rank-one symmetric tensors.

\begin{theorem}[Alexander–Hirschowitz Theorem \cite{alexander1995polynomial}]\label{thm:AH}
    The Veronese variety $V_{d,n}$ is $r$-defective if and only if
    \[
        (d,n,r) = \begin{cases}
            (2, n, r), \ 2 \leq r \leq n;\\
            (3,4,7);\\
            (4,2,5);\\
            (4,3,9);\\
            (4,4,14).
        \end{cases}
    \]
\end{theorem}

Combining this result with \Cref{lemma:upperbound_dimensions}, we obtain the following result for $d \geq 3$. Indeed, the finiteness of the sporadic defective cases for $d \geq 3$ allows us to treat such finitely many cases where the lower bound of \Cref{eq:chain_of_dimensions} 
is not optimal by direct computation. This approach cannot be done for $d = 2$ where \textit{all} secant varieties not filling the ambient space are defective for any $n$.

\begin{corollary}\label{cor:nonhadamarddefect_veronese}
	Let $d \geq 3$ and let $\bfr = (r_1,\ldots,r_m)$. The Hadamard product of secant varieties of Veronese varieties $\sigma_\bfr(V_{d,n})$ are never Hadamard-defective. In particular, the generic $r$th symmetric Hadamard rank of tensors in ${\rm Sym}^d\bbC^{n+1}$ is equal to 
	\[
	   \left\lceil \frac{{n+d \choose d}-n}{(r-1)(n+1)} \right\rceil.
	\]
\end{corollary}
\begin{proof}
    For $(n,d) \not\in \{(2,4),(3,4),(4,3),(4,4)\}$ the result immediately follows from \Cref{lemma:upperbound_dimensions} and \nameref{thm:AH}. The only cases left are the ones corresponding to $(n,d,\bfr)$ for which $(n,d,\sum_k(r_k-1)+1)$ is one of the defective cases and then the lower bound in \Cref{eq:chain_of_dimensions} is not optimal. In these finitely many cases, we check directly that the dimension of $\sigma_\bfr(V_{d,n})$ is equal to the upper bound of \Cref{eq:chain_of_dimensions}. Note that in all these cases, this is equal to the dimension of the ambient space. In order to do that, it is enough to compute the Jacobian of the parametrization of $\sigma_\bfr(V_{d,n})$ and check its rank at a random point. If this is maximal, then we are done. The cases that need to be checked are:
    \[
    \begin{array}{c | c | c | c}
        d & n & \bfr & \# \\
        \hline 
        3 & 4 & \{(2^6),(3,2^5),(4,2^3),(3^2,2^2),(5,2^2),(4,3,2),(3^3),(6,2),(5,3),(4^2)\} & 10\\
        4 & 2 & \{(2^4),(3,2^2),(3^2),(4,2)\} & 4\\
        4 & 3 & \{(2^8),(3,2^6),(4,2^6),(3^2,2^4),(5,2^5),(4,3,2^3),(3^3,2^2),\ldots,(5,3^2),(5^2)\} & 21 \\
        4 & 4 & \{(2^{13}),(3,2^{11}),(4,2^{10}),(3^2,2^9),(5,2^9),(4,3,2^8),\ldots,(6,5,3^2),(6,5^2),(9,6)\} & 100
    \end{array}
    \]
    We checked them computationally using SageMath \cite{sagemath}, finding that they indeed fill the ambient space. The guided computation is stored in a Zenodo page as a Jupyter Notebook \cite{antolini_2025_15837961}. See also \Cref{sec:data}. \qedhere
\end{proof}

\subsubsection{Segre-Veronese products of $\bbP^1$, i.e., the case of binary tensors.}\label{subsec:dimension_binary}
We have a complete classification of defective Segre-Veronese embeddings of $(\bbP^1)^{\times d}$. The case of Segre varieties, i.e., $\mathbf{d} = \mathds{1}$, is due to Catalisano, Geramita and Gimigliano \cite[Theorem 4.1]{catalisano2011secant}. The classification for all Segre-Veronese varieties was given by Laface and Postinghel, see \cite[Theorem 3.1]{laface2013secant}.

\begin{theorem}[Catalisano-Geramita-Gimigliano \cite{catalisano2011secant}; Laface-Postinghel \cite{laface2013secant}]\label{thm:classification_binary}
    The Segre-Veronese variety $SV_{{\bf d},\mathds{1}}$ is $s$-defective if and only if 
    \[
        (\mathbf{d},s) = \begin{cases}
            ((2,2t),2t+1) & \text{ for all } t \geq 1; \\
            ((1,1,2t),2t+1) & \text{ for all } t \geq 1; \\
            ((2,2,2),7); \\
            ((1,1,1,1),3).
        \end{cases}
    \]
\end{theorem}
The result by Catalisano, Geramita and Gimigliano \cite[Theorem 4.1]{catalisano2011secant} on Segre varieties was exploited in \cite[Corollary 26]{MM17:DimensionKronecker} to show that binary Restricted Boltzmann Machines are never defective. Thanks to \Cref{lemma:upperbound_dimensions}, we deduce the following generalization.

\begin{corollary}
\label{cor:generic_binary_tensors}
	Hadamard products of secant varieties of the Segre-Veronese variety of rank-one partially symmetric binary tensors $SV_{{\bf d}, \mathds{1}}$ are never Hadamard-defective for $\mathbf{d} \not\in \{(2,2t),(1,1,2t) ~:~ t \geq 1\}$. In particular, the generic partially symmetric Hadamard rank of tensors in $\mathrm{Sym}^{d_1}\bbC^{2}\otimes \cdots \otimes \mathrm{Sym}^{d_n}\bbC^{2}$ is 
	\[
		\left\lceil \frac{\prod_i(d_i+1)-n}{(r-1)(n+1)} \right\rceil.
	\]
\end{corollary}
\begin{proof}
    For $\mathbf{d}\not\in\{(2,2,2),(1,1,1,1)\}$ the result immediately follows from \Cref{lemma:upperbound_dimensions} and \Cref{thm:classification_binary}. The only cases left are the ones corresponding to $(\mathbf{d},\bfr)$ such that $(\mathbf{d},\sum_k(r_k-1)+1)$ is one of the defective case. In these finitely many cases, we check the dimension directly as explained in the proof of \Cref{cor:nonhadamarddefect_veronese}. The cases that need to be checked are:
    \[
    \begin{array}{c | c | c }
        \mathbf{d} & \bfr & \#\\
        \hline 
        (2,2,2) & \{(2^6),(3,2^5),(4,2^3),(3^2,2^2),(5,2^2),(4,3,2),(3^3),(6,2),(5,3),(4^2)\} & 10 \\
        (1,1,1,1) & \{(2^2)\} & 1
    \end{array}
    \]
    We checked them computationally using SageMath \cite{sagemath}, finding that in all these cases the dimension is the expected one, i.e., is given by the upper bound in \Cref{eq:chain_of_dimensions}. In particular, for the first choice of ${\bf d}$, the expected dimension is the dimension of the ambient space, while for the second choice we obtain an hypersurface. The guided computation is stored in a Zenodo page as a Jupyter Notebook \cite{antolini_2025_15837961}. See also \Cref{sec:data}.\qedhere
\end{proof}

\subsubsection{Segre-Veronese embedded in high degrees}
Recently, it was shown that Segre-Veronese varieties with arbitrary many factors of multidegrees $\bfd = (d_1,\ldots,d_k)$ with $d_1,d_2 \geq 2$ and $d_3,\ldots,d_k \geq 3$ are never defective, see \cite[Theorem 1.2]{ballico2024non}. Similarly as above, these non-defectiveness results combined with \Cref{lemma:upperbound_dimensions} allow us to deduce the following.

\begin{corollary}
	Let ${\bf d} = (d_1, \ldots, d_k)$ with $d_1,d_2 \geq 3$ and $d_3,\ldots,d_k \geq 2$. Let ${\bf n} = (n_1, \ldots, n_k)$. Then, Hadamard products of secant varieties of Segre-Veronese varieties $\sigma_\bfr(SV_{\bf d, n})$ are never Hadamard-defective. In particular, the generic partially symmetric $r$th Hadamard rank of tensors in $\mathrm{Sym}^{d_1}\bbC^{2}\otimes \cdots \otimes \mathrm{Sym}^{d_n}\bbC^{2}$ is equal to 
	\[
		\left\lceil \frac{\prod_i{n_i+d_i \choose d_i}-(n_1+\ldots+n_k)}{(r-1)(n_1+\ldots+n_k+1)} \right\rceil.
	\]
\end{corollary}
\begin{proof}
    It follows from \cite[Theorem 1.2]{ballico2024non} and \Cref{lemma:upperbound_dimensions}.
\end{proof}

\section{Future works}
\label{sec:future} 

In this section, we collect possible directions for future work. 

\subsection{The defective cases} When dealing with Segre-Veronese varieties in \Cref{subsec:dimension_veronese,subsec:dimension_binary}, we left open the cases of quadratic Veronese varieties $V_{2,n}$ and Segre-Veronese varieties $SV_{(2,2t),\mathds{1}}$ and $SV_{(1,1,2t),\mathds{1}}$. This is due to the fact that these varieties are \textit{always} defective and we would have infinitely many cases in which there is a gap between the lower and the upper bound in 
\Cref{eq:chain_of_dimensions}. Hence, we need a different and more direct way to compute the dimension of $\sigma_\bfr(SV_{\bfd,\bfn})$. 

However, in many cases, we computationally verified that the actual dimension is equal to the expected one. We report here a list of these cases. The guided computation is stored in a Zenodo page as a Jupyter Notebook \cite{antolini_2025_15837961}. See also \Cref{sec:data}.

\begin{itemize}
    \item ${\bf d} = (2)$
    \begin{itemize}
        \item $n = 2,\ldots, 15$, for all ${\bf r} = (r_1,r_2)$
        \item $n = 2,\ldots, 12$, for all ${\bf r} = (r_1,r_2,r_3)$
        \item for all ${\bf r} = (r_1,\ldots, r_m)$ such that $\sum_{i = 1}^m r_i \leq T_n$ and for all $n = 2,\ldots, 10$, where
        \begin{center}
        \begin{tabular}{c|c|c|c|c|c|c|c|c|c}
           $n$ & 2 & 3 & 4 & 5 & 6 & 7 & 8 & 9 & 10 \\
           \hline
           $T_n$ & 25 & 25 & 24 & 20 & 15 & 15 & 12 & 12 & 10 
        \end{tabular}
        \end{center}
    \end{itemize}
    \item ${\bf n} = (1,1), \ {\bf d} = (2,2t)$, for all $t = 1,\ldots, 6$
    \item ${\bf n} = (1,1,1), \ {\bf d} = (1,1,2t)$, for all $t = 1, \ldots, 6$
\end{itemize}

Motivated by these computational experiments, we propose the following conjecture.

\begin{conjecture}
\label{conj1}
    For any $\bfd$ and any $\bfn$, Hadamard products of secant varieties of Segre-Veronese varieties $\sigma_\bfr(SV_{\bf d, n})$ are not Hadamard-defective.
\end{conjecture}
Experiments for the case of matrices ($\bfn = (n_1,n_2)$ and $\mathbf{d}=(1,1)$) can also be found in \cite{FOW:MinkowskiHadamard}.

\subsection{Identifiability of Hadamard decompositions} 

Let $X \subset \bbP^N$ be an algebraic variety. For any $\bfr = (r_1,\ldots,r_m)$, consider a generic $p \in \sigma_\bfr(X)=\sigma_{r_1}(X)\star\cdots\star\sigma_{r_m}(X)$:
\begin{center}
\textit{in how many ways can one write $p = q_1\star\cdots\star q_m$ with $q_i \in \langle p_{i,1},\ldots,p_{i,r_i}\rangle\subset \sigma_{r_i}(X)$?} 
\end{center}
In other words, if we consider the map parametrizing the affine cone of the Hadamard product, namely
\[
    \mathfrak{h}_{\bfr}\colon \prod_{i=1}^m\widehat{X}^{\times r_i} \longrightarrow \widehat{\sigma_\bfr(X)}; \quad (p_{i,j})_{\substack{i=1,\ldots,m \\ j = 1,\ldots,r_i}} \mapsto \bigstar_{i=1}^m\left(\sum_{j=1}^{r_i} p_{i,j}\right) , 
\]
\textit{how does $\mathfrak{h}_{\bfr}^{-1}(p)$ look like for a general $p\in \widehat{\sigma_\bfr(X)}$?}

Clearly, there are actions of permutation groups. If $\frak{S}_r$ denotes the permutation on $r$ elements and $\frak{S}(\bfr) \subset \frak{S}_m$ is the subgroup preserving $\bfr$, then we consider the map $\mathfrak{h}_{\bfr}$ up to the action of $\prod_{i=1}^m\frak{S}_{r_i}\times \frak{S}(\bfr)$. We denote by $\widehat{X}^{(\bfr)}$ the quotient of $\prod_{i=1}^m\widehat{X}^{\times r_i}$ by these permutation groups.

If $X$ is a toric variety, we have also the action of $(m-1)$-copies of the full dimensional torus $\bbT_X^{m-1}$:
\[
    (t_1,\ldots,t_{m-1})\cdot (p_{i,j})_{\substack{i=1,\ldots,m\\j=1,\ldots,r_i}} := \begin{cases}
        t_i\cdot p_{i,j} & \text{ for } i =1,\ldots,m-1;\\
        (t_1\cdots t_{m-1})^{-1}\cdot p_{m,j} & \text{ for }i = m.
    \end{cases}
\]
Following a similar approach to the one considered in \cite{oneto2023hadamard}, we consider the quotient of $\mathfrak{h}_\bfr^{-1}(p)$ up to the action of all these groups, i.e., we consider the map 
\[
    \mathcal{H}_\bfr \colon \widehat{X}^{(\bfr)} / \bbT_X^{m-1} \longrightarrow \widehat{\sigma_\bfr(X)}, \ (p_{i,j})_{\substack{i=1,\ldots,m \\ j = 1,\ldots,r_i}} \mapsto \bigstar_{i=1}^m\left(\sum_{j=1}^{r_i} p_{i,j}\right).
\]
Since the dimension of $X^{(\bfr)} / \bbT_X^{m-1}$ is equal to $\sum_i\dim\sigma_{r_i}(X)-(m-1)\dim X$, it follows that whenever $\sigma_\bfr(X)$ has the same dimension, namely in the cases in which it is not Hadamard-defective and it is not overfitting its ambient space, then the general fiber of $\mathcal{H}_\bfr$ is finite. By overfitting, we mean that parameter count $\sum_i\dim\sigma_{r_i}(X)-(m-1)\dim X$ does not strictly exceed the dimension of the ambient space.
Following the standard literature on secant varieties and additive decompositions, if it is a singleton, then we say that $\sigma_\bfr(X)$ is \textit{generically identifiable}. This rises the following natural question.
\begin{question}
    Are the cases of known Hadamard-nondefective and not over-fitting Segre-Veronese varieties shown in \Cref{sec:dimension_Hadamard_powers} generically identifiable? 
\end{question}
In \cite{oneto2023hadamard}, it was proved that this is the case for $m$th Hadamard powers of $r$th secant varieties of Segre varieties as long as $mr$ is below a numerical bound known as the \textit{reshaped Kruskal bound} \cite{chiantini2017effective}. 
{In \cite{pmlr-v238-kant24a}, a certain type of Hadamard product model is considered where each Hadamard factor corresponds to a subset of $\sigma_2 (V_{d,1})$ (fixed mixture weights). They show that for $d$ at least equal to the number of Hadamard factors, this model is generically locally identifiable.}  

\subsection{Equations of Hadamard powers of secant varieties}
\label{sec:geometry-of-Hadamard} 
We have established results on the dimension of several Hadamard powers of secant varieties. 
A natural next question is whether we can determine the polynomial equations that vanish on the model. 
For the Hadamard square of the second secant of the Segre embedding of $\mathbb{P}^1\times \mathbb{P}^1\times \mathbb{P}^1\times \mathbb{P}^1$, \cite{CTY10:Implicitization} observed that it is a hypersurface and obtained descriptions of the polynomial equation that defines it in $\mathbb{P}^{15}$. 
Another case was communicated to us by Yulia Alexandr and Bernd Sturmfels, who computed a degree 18 equation with 2088 monomials that defines $\sigma_2(V_{6,1})^{\star 2}$ as a hypersurface in $\mathbb{P}^6$. This is stored in a Zenodo page \cite{antolini_2025_15837961}. See also \Cref{sec:data}.

\subsection{Lower bounds on the dimension} 

In \Cref{lemma:upperbound_dimensions} we extended \cite[Lemma 25]{MM17:DimensionKronecker} to show that the dimension of a Hadamard product of secant varieties of a toric variety $X$ is always bounded from below by the dimension of the secant variety of $X$ with the same expected dimension. 
Although these two types of varieties are quite different in general \cite{doi:10.1137/140957081}, we showed that they have parametrizations whose Jacobians can be related to one another through a certain parameter scaling procedure. 
To do this, we used that the matrix $\overline{B}_{\mathbf{r}'}$ defining the parametrization of $\sigma_{r_1}(X)\star\cdots\star\sigma_{r_m}(X)$ contains as a sub-matrix the matrix $\overline{I}_R$ defining the parametrization of $\sigma_{R}(X)$. Further characterizing the conditions that allow the application of this parameter scaling procedure in order to lower bound the dimension of one variety by the dimension of another variety would be an interesting future work.

\subsection{Minimal decompositions and universal approximation} 
In \Cref{sec:finiteness_HadamardRanks} we established upper bounds on the generic Hadamard rank; that is, the minimum number of factors so that the Zariski closure of the Hadamard product of the varieties fills the ambient space. 
We also established upper bounds on the Hadamard rank of a point $p$; that is, the minimum number of elements in the Zariski closure of a variety so that their Hadamard product is equal to $p$. 
In this direction, another topic of interest is the minimal number of factors that are needed when one does not take Zariski closures. Concretely, what is the minimum $m$ so that $p$ can be approximated arbitrarily well (in the standard topology of the ambient space of $X$) by a Hadamard product of points in a semi-algebraic set? For instance, we could consider expressions of the form $p_1\star \cdots \star p_m$, where each $p_i$ is a convex combination of $r$ points in $X$. 
This problem is related to the \emph{approximation errors} and \emph{universal approximation} discussed in \cite{MONTUFAR2017531}. Also, can we obtain a semi-algebraic description of a Hadamard product of semi-algebraic sets? This is related to the questions raised in \cite{SM18:MixturesTwoModels} and \cite{sonthalia2023supermodularranksetfunction} about inequalities and the questions raised in \Cref{sec:geometry-of-Hadamard} about the equations that cut the variety of Hadamard powers.

\section*{Data availability}
\label{sec:data}

The computations related to expected defective cases in \Cref{cor:nonhadamarddefect_veronese} and \Cref{cor:generic_binary_tensors} together with some experiments regarding \Cref{conj1} were performed using SageMath \cite{sagemath}. Jupyter Notebooks with the guided computations are provided in the Zenodo page \cite{antolini_2025_15837961}. These include functions to compute parametrizations of Segre-Veronese varieties, their secant varieties and their Hadamard products. This page also contains
the equation of $\sigma_2(V_{6,1})^{\star 2}$ stored both as a SageMath object and as a text file: this was shared to us by Yulia Alexandr and Bernd Sturmfels via private communication.

\bibliographystyle{alphaurl}
\bibliography{Hadamard_ranks.bib}

\newcommand{\etalchar}[1]{$^{#1}$}
\begin{thebibliography}{KYMS{\etalchar{+}}24}

\bibitem[ABO25]{antolini2025algebraic}
Dario Antolini, Edoardo Ballico, and Alessandro Oneto.
\newblock Algebraic surfaces as {Hadamard} products of curves.
\newblock {\em arXiv preprint}, 2025.
\newblock \href {https://arxiv.org/abs/2502.18813} {\path{arXiv:2502.18813}}.

\bibitem[Adl88]{adlandsvik1988varieties}
Bjorn Adlandsvik.
\newblock \href{https://doi.org/10.1515/crll.1988.392.16}{Varieties with an
  extremal number of degenerate higher secant varieties.}
\newblock {\em Journal für die reine und angewandte Mathematik}, 1988.

\bibitem[AH95]{alexander1995polynomial}
James Alexander and André Hirschowitz.
\newblock Polynomial interpolation in several variables.
\newblock {\em Journal of Algebraic Geometry}, 4(4):201--222, 1995.

\bibitem[AMO25]{antolini_2025_15837961}
Dario Antolini, Guido Montúfar, and Alessandro Oneto.
\newblock {Supplementary code to ``Hadamard ranks of algebraic varieties''}.
\newblock {\em Zenodo}, September 2025.
\newblock \href {https://doi.org/10.5281/zenodo.15837960}
  {\path{doi:10.5281/zenodo.15837960}}.

\bibitem[Bal24]{ballico2024non}
Edoardo Ballico.
\newblock {\href{https://doi.org/10.1007/s00209-024-03573-x}{On the
  non-defectivity of Segre--Veronese embeddings}}.
\newblock {\em Mathematische Zeitschrift}, 308(1):6, 2024.

\bibitem[Bal25]{BALLICO}
Edoardo Ballico.
\newblock \href{https://doi.org/10.1016/j.jpaa.2024.107853}{Projective surfaces
  not as {H}adamard products and the dimensions of the {H}adamard joins}.
\newblock {\em Journal of Pure and Applied Algebra}, 229(1):107853, 2025.

\bibitem[BC22]{bocci2022hadamard}
Cristiano Bocci and Enrico Carlini.
\newblock \href{https://doi.org/10.1016/j.jpaa.2022.107078}{Hadamard products
  of hypersurfaces}.
\newblock {\em Journal of Pure and Applied Algebra}, 226(11):107078, 2022.

\bibitem[BC24]{bocci2024hadamard}
Cristiano Bocci and Enrico Carlini.
\newblock {\em \href{https://doi.org/10.1007/978-3-031-54263-3}{{Hadamard
  Products of Projective Varieties}}}.
\newblock Springer, 2024.

\bibitem[BCC{\etalchar{+}}18]{BCCGO:Hitchhiker}
Alessandra Bernardi, Enrico Carlini, Maria Catalisano, Alessandro Gimigliano,
  and Alessandro Oneto.
\newblock \href{https://doi.org/10.3390/math6120314}{{The Hitchhiker Guide to:
  Secant Varieties and Tensor Decomposition}}.
\newblock {\em Mathematics}, 6(12):314, 2018.

\bibitem[BCK16]{BCK17}
Cristiano Bocci, Enrico Carlini, and Joe Kileel.
\newblock \href{https://doi.org/10.1016/j.jalgebra.2015.10.008}{Hadamard
  products of linear spaces}.
\newblock {\em Journal of Algebra}, 448:595--617, 2016.

\bibitem[BT15]{blekherman2015maximum}
Grigoriy Blekherman and Zach Teitler.
\newblock \href{https://doi.org/10.1007/s00208-014-1150-3}{On maximum, typical
  and generic ranks}.
\newblock {\em Mathematische Annalen}, 362:1021--1031, 2015.

\bibitem[CGG11]{catalisano2011secant}
Maria~Virginia Catalisano, Anthony~V.\ Geramita, and Alessandro Gimigliano.
\newblock \href{https://doi.org/10.1090/S1056-3911-10-00537-0}{Secant varieties
  of $\mathbb{P}^1\times\cdots\times\mathbb{P}^1$ ($n$-times) are NOT defective
  for $n\geq 5$}.
\newblock {\em Journal of Algebraic Geometry}, 20(2):295--327, 2011.

\bibitem[CGM24]{ciaperoni2024hadamard}
Martino Ciaperoni, Aristides Gionis, and Heikki Mannila.
\newblock \href{https://doi.org/10.1007/s10618-024-01033-y}{The Hadamard
  decomposition problem}.
\newblock {\em Data Mining and Knowledge Discovery}, 38(4):2306--2347, 2024.

\bibitem[CLS11]{cox2011toric}
David~A. Cox, John~B. Little, and Henry~K. Schenck.
\newblock {\em \href{https://doi.org/10.1090/gsm/124}{Toric Varieties}}, volume
  124.
\newblock American Mathematical Soc., 2011.

\bibitem[CMS10]{CMS10:GeometryRBM}
Mar{\'\i}a~Ang{\'e}lica Cueto, Jason Morton, and Bernd Sturmfels.
\newblock \href{https://doi.org/10.1090/conm/516/10172}{Geometry of the
  restricted {B}oltzmann machine}.
\newblock {\em Algebraic Methods in Statistics and Probability}, 516:135--153,
  2010.

\bibitem[COV17]{chiantini2017effective}
Luca Chiantini, Giorgio Ottaviani, and Nick Vannieuwenhoven.
\newblock \href{https://doi.org/10.1137/16M1090132}{Effective criteria for
  specific identifiability of tensors and forms}.
\newblock {\em SIAM Journal on Matrix Analysis and Applications},
  38(2):656--681, 2017.

\bibitem[CTY10]{CTY10:Implicitization}
Mar{\'\i}a~Ang{\'e}lica Cueto, Enrique~A. Tobis, and Josephine Yu.
\newblock \href{https://doi.org/10.1016/j.jsc.2010.06.011}{An implicitization
  challenge for binary factor analysis}.
\newblock {\em Journal of Symbolic Computation}, 45(12):1296--1315, 2010.

\bibitem[FOW17]{FOW:MinkowskiHadamard}
Netanel Friedenberg, Alessandro Oneto, and Robert~L. Williams.
\newblock \href{https://doi.org/10.1007/978-1-4939-7486-3_7}{{Minkowski sums
  and Hadamard products of algebraic varieties}}.
\newblock In {\em Combinatorial algebraic geometry}, pages 133--157. Springer,
  2017.

\bibitem[Hit27]{hitchcock1927expression}
Frank~L. Hitchcock.
\newblock \href{https://doi.org/10.1002/sapm192761164}{The expression of a
  tensor or a polyadic as a sum of products}.
\newblock {\em Journal of Mathematics and Physics}, 6(1-4):164--189, 1927.

\bibitem[JKK17]{jensen2017finding}
Anders Jensen, Thomas Kahle, and Lukas Katth{\"a}n.
\newblock \href{https://doi.org/10.1186/s40687-017-0106-0}{Finding binomials in
  polynomial ideals}.
\newblock {\em Research in the Mathematical Sciences}, 4:1--10, 2017.

\bibitem[KYMS{\etalchar{+}}24]{pmlr-v238-kant24a}
Manav Kant, Eric Y.~Ma, Andrei Staicu, Leonard J.~Schulman, and Spencer Gordon.
\newblock
  {\href{https://proceedings.mlr.press/v238/kant24a.html}{Identifiability of
  Product of Experts Models}}.
\newblock In Sanjoy Dasgupta, Stephan Mandt, and Yingzhen Li, editors, {\em
  Proceedings of The 27th International Conference on Artificial Intelligence
  and Statistics}, volume 238 of {\em Proceedings of Machine Learning
  Research}, pages 4492--4500. PMLR, 02--04 May 2024.

\bibitem[Lan12]{Lan12:Book}
Joseph~M. Landsberg.
\newblock \href{https://doi.org/10.1090/gsm/128}{Tensors: geometry and
  applications}.
\newblock {\em Representation theory}, 381(402):3, 2012.

\bibitem[LMR22]{laface2022secant}
Antonio Laface, Alex Massarenti, and Rick Rischter.
\newblock \href{https://doi.org/10.4171/RMI/1336}{{On secant defectiveness and
  identifiability of Segre--Veronese varieties}}.
\newblock {\em Revista Matem{\'a}tica Iberoamericana}, 38(5):1605--1635, 2022.

\bibitem[LP13]{laface2013secant}
Antonio Laface and Elisa Postinghel.
\newblock \href{https://doi.org/10.1007/s00208-012-0890-1}{{Secant varieties of
  Segre-Veronese embeddings of $(\mathbb{P}^1)^r$}}.
\newblock {\em Mathematische Annalen}, 356(4):1455--1470, 2013.

\bibitem[MM15a]{doi:10.1137/140957081}
Guido Mont\'{u}far and Jason Morton.
\newblock \href{https://doi.org/10.1137/140957081}{When Does a Mixture of
  Products Contain a Product of Mixtures?}
\newblock {\em SIAM Journal on Discrete Mathematics}, 29(1):321--347, 2015.

\bibitem[MM15b]{montufar2015discrete}
Guido Mont{\'u}far and Jason Morton.
\newblock {\href{https://www.jmlr.org/papers/v16/montufar15a.html}{Discrete
  restricted {B}oltzmann machines}}.
\newblock {\em J. Mach. Learn. Res.}, 16(1):653--672, 2015.

\bibitem[MM17]{MM17:DimensionKronecker}
Guido Mont{\'u}far and Jason Morton.
\newblock \href{https://doi.org/10.1137/16M1077489}{Dimension of marginals of
  {K}ronecker product models}.
\newblock {\em SIAM Journal on Applied Algebra and Geometry}, 1(1):126--151,
  2017.

\bibitem[Mon16]{Mon16}
Guido Mont{\'u}far.
\newblock \href{https://doi.org/10.1007/978-3-319-97798-0_4}{Restricted
  {B}oltzmann {M}achines: {I}ntroduction and {R}eview}.
\newblock In {\em Information Geometry and its Applications IV}, pages 75--115.
  Springer, 2016.

\bibitem[MPS25]{maazouz2024spinor}
Yassine~El Maazouz, Ana{\"e}lle Pfister, and Bernd Sturmfels.
\newblock Spinor-helicity varieties.
\newblock {\em To appear in SIAM Journal on Applied Algebra and Geometry},
  2025.
\newblock \href {https://arxiv.org/abs/2406.17331} {\path{arXiv:2406.17331}}.

\bibitem[MR17]{MONTUFAR2017531}
Guido Montúfar and Johannes Rauh.
\newblock \href{https://doi.org/10.1016/j.ijar.2016.09.003}{Hierarchical models
  as marginals of hierarchical models}.
\newblock {\em International Journal of Approximate Reasoning}, 88:531--546,
  2017.

\bibitem[MS15]{sturmfelsmaclagan2015}
Diane Maclagan and Bernd Sturmfels.
\newblock {\em \href{https://doi.org/10.1090/gsm/161}{Introduction to Tropical
  Geometry}}.
\newblock Graduate Studies in Mathematics. American Mathematical Society, 2015.

\bibitem[MSUZ16]{michalek2016exponential}
Mateusz Micha{\l}ek, Bernd Sturmfels, Caroline Uhler, and Piotr Zwiernik.
\newblock \href{https://doi.org/10.1112/plms/pdv066}{Exponential varieties}.
\newblock {\em Proceedings of the London Mathematical Society}, 112(1):27--56,
  2016.

\bibitem[OV23]{oneto2023hadamard}
Alessandro Oneto and Nick Vannieuwenhoven.
\newblock {Hadamard-Hitchcock decompositions: identifiability and computation}.
\newblock {\em arXiv preprint}, 2023.
\newblock \href {https://arxiv.org/abs/2308.06597} {\path{arXiv:2308.06597}}.

\bibitem[OV25]{oneto2025ranks}
Alessandro Oneto and Emanuele Ventura.
\newblock \href{https://doi.org/10.1007/s40574-025-00472-9}{Ranks of tensors:
  geometry and applications}.
\newblock {\em Bollettino dell'Unione Matematica Italiana}, pages 1--28, 2025.

\bibitem[{Sag}23]{sagemath}
{The} {Sage Developers}.
\newblock {\em \href{https://www.sagemath.org}{{S}ageMath, the {S}age
  {M}athematics {S}oftware {S}ystem ({V}ersion 9.7)}}, 2023.

\bibitem[SM18]{SM18:MixturesTwoModels}
Anna Seigal and Guido Mont{\'u}far.
\newblock \href{https://doi.org/10.18409/jas.v9i1.90}{Mixtures and products in
  two graphical models}.
\newblock {\em Journal of Algebraic Statistics}, 9(1), 2018.

\bibitem[SSM23]{sonthalia2023supermodularranksetfunction}
Rishi Sonthalia, Anna Seigal, and Guido Mont\'ufar.
\newblock Supermodular rank: Set function decomposition and optimization.
\newblock {\em arXiv preprint}, 2023.
\newblock \href {https://arxiv.org/abs/2305.14632} {\path{arXiv:2305.14632}}.

\bibitem[Syl51]{sylvester1851lx}
James~Joseph Sylvester.
\newblock \href{https://doi.org/10.1080/14786445108645733}{{LX}. {O}n a
  remarkable discovery in the theory of canonical forms and of
  hyperdeterminants}.
\newblock {\em The London, Edinburgh, and Dublin Philosophical Magazine and
  Journal of Science}, 2(12):391--410, 1851.

\bibitem[TBC23]{taveira2023nondefectivity}
Alexander Taveira~Blomenhofer and Alex Casarotti.
\newblock Nondefectivity of invariant secant varieties.
\newblock {\em arXiv preprint}, 2023.
\newblock \href {https://arxiv.org/abs/2312.12335} {\path{arXiv:2312.12335}}.

\bibitem[Zie12]{ziegler2012lectures}
Günter~Matthias Ziegler.
\newblock {\em \href{https://doi.org/10.1007/978-1-4613-8431-1}{Lectures on
  Polytopes}}.
\newblock Graduate Texts in Mathematics. Springer New York, 2012.

\end{thebibliography}

\end{document}